\documentclass{article}

\usepackage[english]{babel}

\usepackage[letterpaper,top=2cm,bottom=2cm,left=3cm,right=3cm,marginparwidth=1.75cm]{geometry}

\usepackage{titlesec}

\titleformat*{\section}{\large\bfseries}

\usepackage{caption}

\usepackage[T1,T2A]{fontenc}
\usepackage[utf8]{inputenc}
\usepackage{amsmath,amssymb,amsthm,verbatim}
\usepackage{graphicx}
\usepackage{hyperref}
\usepackage{multicol} 
\usepackage{amsmath}
\usepackage{tikz}
\usepackage{mathdots}
\usepackage{yhmath}
\usepackage{cancel}
\usepackage{color}
\usepackage{siunitx}
\usepackage{array}
\usepackage{multirow}
\usepackage{amssymb}
\usepackage{enumerate}
\usepackage{gensymb}
\usepackage{tabularx}
\usepackage{booktabs}
\usetikzlibrary{fadings}
\usepackage[thinlines]{easytable}
\usepackage{float}

\theoremstyle{plain}
\newtheorem{theorem}{Theorem}[section]

\newtheorem{lemma}[theorem]{Lemma}
\newtheorem{cor}[theorem]{Corollary}
\newtheorem{prop}[theorem]{Proposition}
\newtheorem{claim}[theorem]{Claim}

\theoremstyle{definition}
\newtheorem{defn}[theorem]{Definition}

\title{\Large \textbf{Cardinal invariants related to density}}
\author{\normalsize David Valderrama}
\date{}

\begin{document}

\renewcommand{\abstractname}{\vspace{-\baselineskip}}

\maketitle

\begin{abstract}
\noindent \textit{Abstract.} We investigate some variants of the splitting, reaping, and independence numbers defined using asymptotic density. Specifically, we give a proof of Con($\mathfrak{i}<\mathfrak{s}_{1/2}$), Con($\mathfrak{r}_{1/2}<\mathfrak{b}$) and Con($\mathfrak{i}_*<2^{\aleph_0}$). This answers two questions raised in \cite{Halway2018}. Besides, we prove the consistency of $\mathfrak{s}_{1/2}^{\infty} < $ non$(\mathcal{E})$ and cov$(\mathcal{E}) < \mathfrak{r}_{1/2}^{\infty}$, where $\mathcal{E}$ is the $\sigma$-ideal generated by closed sets of measure zero.

\hspace{2mm}


\hspace{1mm}


\end{abstract}

\section{Introduction}

Given $S,X \in [\omega]^\omega$, we say that \textit{$S$ splits $X$} if $|X \cap S| = |X \setminus S | = \aleph_0$. The \textit{splitting number} $\mathfrak{s}$ is the least size of a splitting family $\mathcal{S}\subseteq[\omega]^\omega$, that is, every infinite subset of $\omega$ is split by a member of $\mathcal{S}$. For $X \in [\omega]^\omega$ and $0<n<\omega$, define the \textit{initial density} of $X$ up to $n$ as $d_n(X)=\frac{|X\cap n|}{n}$. In case of convergence of $d_n(X)$, call $d(X)=\lim_{n \to \infty} d_n(X)$ the \textit{asymptotic density} or just the \textit{density} of $X$. In \cite{Halway2018} the following variants of $\mathfrak{s}$ were introduced by using the notion of asymptotic density to characterize different intersection properties of infinite sets. 

\begin{defn} Let $S,X \in [\omega]^{\omega}$. 
\begin{itemize}
    \item $S$ \textit{bisects} $X$, written as $S |_{1/2} X$, if
    \[
        \lim_{n \to \infty} \frac{|S \cap X \cap n|}{|X\cap n|}= \lim_{n \to \infty} \frac{d_n(S\cap X)}{d_n(X)}= \frac{1}{2}
    \]
    \item For $0< \epsilon< 1/2$, \textit{$S$ $\epsilon$-almost bisects $X$}, written as $S \, |_{1/2 \pm \epsilon} \, X$, if for all but finitely many $n \in \omega$ we have
    \[
     \frac{|S \cap X \cap n|}{|X\cap n|}= \frac{d_n(S\cap X)}{d_n(X)} \in \left(\frac{1}{2}- \epsilon, \frac{1}{2}+ \epsilon\right)
    \]
    \item \textit{$S$ weakly bisects $X$}, written as $S \, |_{1/2}^{w} \, X$, if for any $\epsilon>0$, for infinitely many $n\in \omega$ we have
    \[
            \frac{|S \cap X \cap n|}{|X\cap n|}= \frac{d_n(S\cap X)}{d_n(X)} \in \left(\frac{1}{2}- \epsilon, \frac{1}{2}+ \epsilon\right)
    \]
    \item \textit{$S$ cofinally bisects $X$}, written as $S \, |_{1/2}^{\infty} \, X$, if for infinitely many $n\in \omega$ we have
    \[
            \frac{|S \cap X \cap n|}{|X\cap n|}= \frac{d_n(S\cap X)}{d_n(X)} = \frac{1}{2}
    \]
\end{itemize}
    
\end{defn}

\begin{defn}\label{variants of s}
 We say a family $\mathcal{S}$ of infinite sets is
\[
        \left\{ \begin{array}{c}
           \text{bisecting}   \\
           \text{$\epsilon$-almost bisecting}   \\
           \text{weakly bisecting}   \\
           \text{cofinally bisecting}    
        \end{array} \right.
    \]
    if for each $X \in [\omega]^\omega$ there is some $S \in \mathcal{S}$ such that
    \[
        \left\{ \begin{array}{c}
           \text{$S$ bisects $X$}   \\
           \text{$S$ $\epsilon$-almost bisects $X$}   \\
           \text{$S$ weakly bisects $X$}   \\
           \text{$S$ cofinally bisects $X$}    
        \end{array} \right.
    \]
    and denote the least cardinality of such a family by  $\mathfrak{s}_{1/2}, \mathfrak{s}_{1/2 \pm \epsilon}, \mathfrak{s}_{1/2}^{w}, \mathfrak{s}_{1/2}^{\infty}$, respectively.    
\end{defn}

Let $\mathcal{N}$ denote the ideal of Lebesgue null sets, $\mathcal{M}$ the ideal of meager sets, $\mathfrak{d}$ the dominating number, and $\mathfrak{b}$ the unbounding number (see \cite{Blass2010} for information on these concepts). Figure \ref{Inequalities between new variants of s} shows the
relations already proved between these new cardinals and other well-known cardinal characteristics.  It
also shows their behavior in some models obtained using classical forcing notions. The proofs of these results are in \cite{Halway2018}.

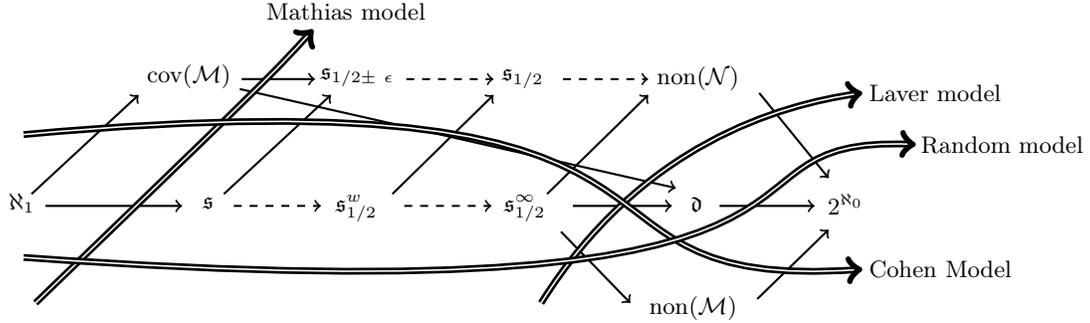
\begin{figure}[h!]
    \centering

\tikzset{every picture/.style={line width=0.75pt}} 

\begin{tikzpicture}[x=0.75pt,y=0.75pt,yscale=-1,xscale=1]

\draw [thick,->]   (80,131) -- (149,131) ;
\draw [thick,->]   (346,131) -- (396,131) ;
\draw [thick,->]   (340,144) -- (375,180) ;
\draw [thick,->]  (420,131) -- (468,131) ;
\draw [thick,->]   (440,74) -- (475,117) ;
\draw [thick,->]   (439,178) -- (475,143) ;
\draw [thick,->]   (178,72) -- (398,122) ;
\draw [dashed,->] (175,131) -- (219,131) ;
\draw  [dashed,->]  (261,131) -- (305,131) ;
\draw  [dashed,->]  (262,67) -- (306,67) ;
\draw  [dashed,->]  (341,67) -- (384,67) ;
\draw [thick,->]   (179,67) -- (216,67) ;
\draw  [thick,->]  (170,125) -- (224,75) ;
\draw  [thick,->]  (73,125) -- (127,75) ;
\draw [thick,->]   (255,125) -- (306,75) ;
\draw  [thick,->]  (333,125) -- (384,75) ;

\draw [line width=1][double,->]    (75,180) .. controls (195,64.09) .. (215,42) ;

\draw [line width=1][double,->]    (330,180) .. controls (380,100) and (450,80) .. (492,74) ;

\draw [line width=1][double,->]    (69,95) .. controls (450,60) and (320,175)..(492,163) ;

\draw [line width=1][double,->]   (69,157) .. controls (530,190) and (400,95) .. (519,100) ;

\draw (60,123) node [anchor=north west][inner sep=0.75pt]    {\small $\aleph _{1}$};
\draw (158,125) node [anchor=north west][inner sep=0.75pt]    {\small $\mathfrak{s} $};
\draw (225,124) node [anchor=north west][inner sep=0.75pt]    {\small $\mathfrak{s}_{1/2}^{w}$};
\draw (310,124) node [anchor=north west][inner sep=0.75pt]    {\small $\mathfrak{s}_{1/2}^{\infty }$};
\draw (404,125) node [anchor=north west][inner sep=0.75pt]    {\small $\mathfrak{d}$};
\draw (473.13,123) node [anchor=north west][inner sep=0.75pt]    {\small $2^{\aleph _{0}}$};
\draw (383,174.82) node [anchor=north west][inner sep=0.75pt]    {\small non$(\mathcal{M})$};
\draw (387,66.82) node [anchor=west] [inner sep=0.75pt]    {\small non$(\mathcal{N})$};
\draw (309,66.82) node [anchor=west] [inner sep=0.75pt]    {\small $\mathfrak{s}_{1/2}$};
\draw (217.5,66) node [anchor=west] [inner sep=0.75pt]    {\small $\mathfrak{s}_{1/2\pm \ \epsilon } \ $};
\draw (175,65.82) node [anchor=east] [inner sep=0.75pt]    {\small cov$(\mathcal{M})$};
\draw (190,33) node [anchor=west] [inner sep=0.75pt]  [font=\small,color={rgb, 255:red, 0; green, 0; blue, 0 }  ,opacity=1 ]  { Mathias model};
\draw (494,74) node [anchor=west] [inner sep=0.75pt]  [font=\small,color={rgb, 255:red, 0; green, 0; blue, 0 }  ,opacity=1 ]  {Laver model};
\draw (494,163) node [anchor=west] [inner sep=0.75pt]  [font=\small,color={rgb, 255:red, 0; green, 0; blue, 0 }  ,opacity=1 ]  {Cohen Model};
\draw (520,100) node [anchor=west] [inner sep=0.75pt]  [font=\small,color={rgb, 255:red, 0; green, 0; blue, 0 }  ,opacity=1 ]  {Random model};

\end{tikzpicture}
    \caption{\small Inequalities between $\mathfrak{s}_{1/2}, \mathfrak{s}_{1/2 \pm \epsilon}, \mathfrak{s}_{1/2}^{w}, \mathfrak{s}_{1/2}^{\infty}$ and other well-known cardinal characteristics. $\dashrightarrow$ means $\leq$, and $\rightarrow$ means $\leq$, consistently $<$.}
    \label{Inequalities between new variants of s}
\end{figure}

The cardinals defined in \ref{variants of s} also have their dual version. Recall that the \textit{reaping number} $\mathfrak{r}$ is the dual of $\mathfrak{s}$, and it is the least size of a reaping family $\mathcal{R} \subseteq [\omega]^\omega$, that is, no single infinite set splits all members of $\mathcal{R}$. 

\begin{defn}\label{variants of r}
 We say a family $\mathcal{R}$ of infinite sets is
\[
        \left\{ \begin{array}{c}
           \text{$1/2$-reaping}   \\
           \text{$\epsilon$-almost $1/2$-reaping}   \\
           \text{weakly $1/2$-reaping}   \\
           \text{cofinally $1/2$-reaping}    
        \end{array} \right.
    \]
    if there is no $X \in [\omega]^\omega$ such that for all $R \in \mathcal{R}$  
    \[
        \left\{ \begin{array}{c}
           \text{$X$ bisects $R$}   \\
           \text{$X$ $\epsilon$-almost bisects $R$}   \\
           \text{$X$ weakly bisects $R$}   \\
           \text{$X$ cofinally bisects $R$}    
        \end{array} \right.
    \]
    and denote the least cardinality of such a family by  $\mathfrak{r}_{1/2}, \mathfrak{r}_{1/2 \pm \epsilon}, \mathfrak{r}_{1/2}^{w}, \mathfrak{r}_{1/2}^{\infty}$, respectively.    
\end{defn}

Figure \ref{dual diagram} shows the inequalities already proved between these new variants of $\mathfrak{r}$ and other well-known cardinal characteristics. The proofs of these results are also in \cite{Halway2018}.

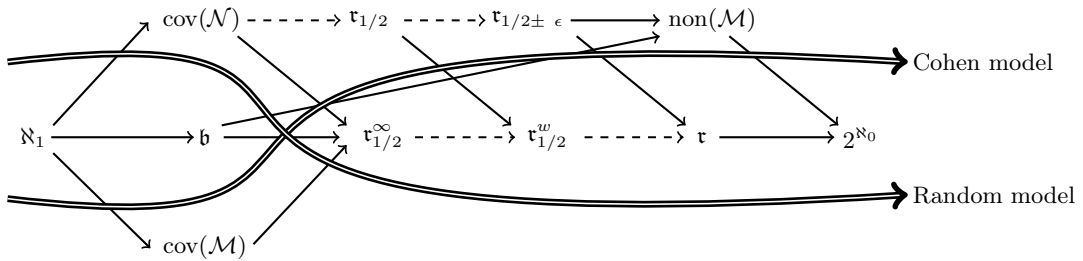
\begin{figure}[H]
    \centering
    
\tikzset{every picture/.style={line width=0.75pt}} 

\begin{tikzpicture}[x=0.75pt,y=0.75pt,yscale=-1,xscale=1]

\draw  [thick,->]   (83,541) -- (155,541) ;
\draw  [dashed,->]  (352,541) -- (402,541) ;
\draw  [thick,->]  (185,595) -- (232,545) ;
\draw [thick,->]  (420,541) -- (478,541) ;
\draw  [thick,->]  (170,541) -- (228.5,541) ;
\draw [dashed,->]   (266.5,541) -- (317,541) ;
\draw  [dashed,->] (261,482) -- (302,482) ;
\draw  [thick,->]  (345,482) -- (390,482) ;
\draw  [dashed,->] (182,482) -- (229,482) ;
\draw  [thick,->]  (84,549) -- (133,597) ;
\draw  [thick,->]  (84,535) -- (132,483) ;
\draw  [thick,->]  (425,490) -- (480,535) ;
\draw  [thick,->]  (348,490) -- (403,535) ;
\draw  [thick,->]  (177,490) -- (232,535) ;
\draw  [thick,->]  (260,490) -- (315,535) ;
\draw  [thick,->]  (169,535) -- (390,490) ;

\draw [line width=1][double,->]   (61,572) .. controls (300,600) and (50,475) .. (515,503) ;

\draw [line width=1][double,->]   (61,503) .. controls (300,475) and (40,600) .. (515,570) ;

\draw (82,541) node [anchor=east][inner sep=0.75pt]    {\small $\aleph _{1}$};
\draw (165,541) node [anchor=east][inner sep=0.75pt]    {\small $\mathfrak{b} $};
\draw (345,541) node [anchor=east] [inner sep=0.75pt]    {\small $\mathfrak{r}_{1/2}^{w}$};
\draw (262,541) node [anchor=east] [inner sep=0.75pt]    {\small $\mathfrak{r}_{1/2}^{\infty }$};
\draw (415,541) node [anchor=east][inner sep=0.75pt]    {\small $\mathfrak{r} $};
\draw (502,541) node [anchor=east][inner sep=0.75pt]    {\small $2^{\aleph_{0}}$};
\draw (138,589) node [anchor=north west][inner sep=0.75pt]    {\small cov$(\mathcal{M})$};
\draw (180,482) node [anchor=east] [inner sep=0.75pt]    {\small cov$(\mathcal{N})$};
\draw (255,482) node [anchor=east] [inner sep=0.75pt]    {\small $\mathfrak{r}_{1/2}$};
\draw (347,482) node [anchor=east] [inner sep=0.75pt]    {\small $\mathfrak{r}_{1/2\pm \ \epsilon } \ $};
\draw (440,482) node [anchor=east] [inner sep=0.75pt]    {\small non$(\mathcal{M})$};
\draw (516,503) node [anchor=west] [inner sep=0.75pt]  [font=\small,color={rgb, 255:red, 0; green, 0; blue, 0 }  ,opacity=1 ]  {Cohen model};
\draw (516,570) node [anchor=west] [inner sep=0.75pt]  [font=\small,color={rgb, 255:red, 0; green, 0; blue, 0 }  ,opacity=1 ]  {Random model};

\end{tikzpicture}
    \caption{\small Dual diagram. $\dashrightarrow$ means $\leq$, and $\rightarrow$ means $\leq$, consistently $<$. }
    \label{dual diagram}
\end{figure}

It is consistently true that $\mathfrak{s}_{1/2}< \mathfrak{d}$ since this holds in any model of Con$($non$(\mathcal{N})< \mathfrak{d})$, for example the Laver model. However, the problem about Con($\mathfrak{d}<\mathfrak{s}_{1/2}$) remained open until it was solved recently in \cite{farkas2023halfway}. They also proved the dual problem  Con$(\mathfrak{r}_{1/2} < \mathfrak{b})$. However, different methods were used to prove each result. Independently, we demonstrated that the same inequalities hold using a different and shorter argument. This is developed in section \ref{sigma centered forcing}. We use just one method to prove Con($\mathfrak{d}<\mathfrak{s}_{1/2}$),  Con$(\mathfrak{r}_{1/2} < \mathfrak{b})$, and even a stronger result which is Con($\mathfrak{i}<\mathfrak{s}_{1/2}$). Recall that $\mathfrak{i}$ stands for the independence number (see below), and it is known that  $\mathfrak{d}\leq \mathfrak{i}$.

We also investigate relations between these new cardinal characteristics and the cardinal characteristics associated with the ideal $\mathcal{E}$, which is the $\sigma$-ideal generated by closed sets of measure zero. In section \ref{Hechler forcing} we show that $\mathfrak{s}_{1/2}^{\infty} \leq $ non$(\mathcal{E})$ and cov$(\mathcal{E}) \leq \mathfrak{r}_{1/2}^{\infty}$. Besides, we also prove that the strict inequalities are consistent using the Hechler and Dual Hechler model. By the results in sections \ref{sigma centered forcing} and \ref{Hechler forcing}, we know the behavior of these cardinals in these models (see figure \ref{Behvaior in the Hechler model}).

\begin{figure}[h!]
    \centering

\tikzset{every picture/.style={line width=0.75pt}} 

\begin{tikzpicture}[x=0.75pt,y=0.75pt,yscale=-1,xscale=1]

\draw [thick,->]   (80,131) -- (149,131) ;
\draw [thick,->]   (346,131) -- (396,131) ;
\draw [thick,->]   (340,144) -- (375,180) ;
\draw [thick,->]  (420,131) -- (468,131) ;
\draw [thick,->]   (440,74) -- (475,117) ;
\draw [thick,->]   (439,178) -- (475,143) ;
\draw [thick,->]   (178,72) -- (398,122) ;
\draw [dashed,->] (175,131) -- (219,131) ;
\draw  [dashed,->]  (261,131) -- (305,131) ;
\draw  [dashed,->]  (262,67) -- (306,67) ;
\draw  [dashed,->]  (341,67) -- (384,67) ;
\draw [thick,->]   (179,67) -- (216,67) ;
\draw  [thick,->]  (170,125) -- (224,75) ;
\draw  [thick,->]  (73,125) -- (127,75) ;
\draw [thick,->]   (255,125) -- (306,75) ;
\draw  [thick,->]  (333,125) -- (384,75) ;

\draw [line width=1][double,->]    (200,42) .. controls (180,150) and (250,80).. (404,100)  .. controls (450,110) and (420,170).. (492,163);

\draw [line width=1][double,->]    (60,115) .. controls (320,110) .. (350,125) .. controls (400,160) and (360,170) .. (350,170);

\draw (60,123) node [anchor=north west][inner sep=0.75pt]    {\small $\aleph _{1}$};
\draw (158,125) node [anchor=north west][inner sep=0.75pt]    {\small $\mathfrak{s} $};
\draw (225,124) node [anchor=north west][inner sep=0.75pt]    {\small $\mathfrak{s}_{1/2}^{w}$};
\draw (310,124) node [anchor=north west][inner sep=0.75pt]    {\small $\mathfrak{s}_{1/2}^{\infty }$};
\draw (404,125) node [anchor=north west][inner sep=0.75pt]    {\small $\mathfrak{d}$};
\draw (473.13,123) node [anchor=north west][inner sep=0.75pt]    {\small $2^{\aleph _{0}}$};
\draw (383,174.82) node [anchor=north west][inner sep=0.75pt]    {\small non$(\mathcal{M})$};
\draw (387,66.82) node [anchor=west] [inner sep=0.75pt]    {\small non$(\mathcal{N})$};
\draw (309,66.82) node [anchor=west] [inner sep=0.75pt]    {\small $\mathfrak{s}_{1/2}$};
\draw (217.5,66) node [anchor=west] [inner sep=0.75pt]    {\small $\mathfrak{s}_{1/2\pm \ \epsilon } \ $};
\draw (175,65.82) node [anchor=east] [inner sep=0.75pt]    {\small cov$(\mathcal{M})$};
\draw (494,163) node [anchor=west] [inner sep=0.75pt]  [font=\small,color={rgb, 255:red, 0; green, 0; blue, 0 }  ,opacity=1 ]  {Dual Hechler model};
\draw (270,170) node [anchor=west] [inner sep=0.75pt]  [font=\small,color={rgb, 255:red, 0; green, 0; blue, 0 }  ,opacity=1 ]  {Hechler model};

\end{tikzpicture}

\hspace{2mm}

\tikzset{every picture/.style={line width=0.75pt}} 

\begin{tikzpicture}[x=0.75pt,y=0.75pt,yscale=-1,xscale=1]

\draw  [thick,->]   (83,541) -- (155,541) ;
\draw  [dashed,->]  (352,541) -- (402,541) ;
\draw  [thick,->]  (185,595) -- (232,545) ;
\draw [thick,->]  (420,541) -- (478,541) ;
\draw  [thick,->]  (170,541) -- (228.5,541) ;
\draw [dashed,->]   (266.5,541) -- (317,541) ;
\draw  [dashed,->] (261,482) -- (302,482) ;
\draw  [thick,->]  (345,482) -- (390,482) ;
\draw  [dashed,->] (182,482) -- (229,482) ;
\draw  [thick,->]  (84,549) -- (133,597) ;
\draw  [thick,->]  (84,535) -- (132,483) ;
\draw  [thick,->]  (425,490) -- (480,535) ;
\draw  [thick,->]  (348,490) -- (403,535) ;
\draw  [thick,->]  (177,490) -- (232,535) ;
\draw  [thick,->]  (260,490) -- (315,535) ;
\draw  [thick,->]  (169,535) -- (390,490) ;

\draw [line width=1][double,->]   (215,600) .. controls (200,550) and (100,485) .. (515,510) ;

\draw [line width=1][double,->]   (75,580) .. controls (150,560) and (100,485) .. (350,495) .. controls (375,495) .. (370,461) ;

\draw (82,541) node [anchor=east][inner sep=0.75pt]    {\small $\aleph _{1}$};
\draw (165,541) node [anchor=east][inner sep=0.75pt]    {\small $\mathfrak{b} $};
\draw (345,541) node [anchor=east] [inner sep=0.75pt]    {\small $\mathfrak{r}_{1/2}^{w}$};
\draw (262,541) node [anchor=east] [inner sep=0.75pt]    {\small $\mathfrak{r}_{1/2}^{\infty }$};
\draw (415,541) node [anchor=east][inner sep=0.75pt]    {\small $\mathfrak{r} $};
\draw (502,541) node [anchor=east][inner sep=0.75pt]    {\small $2^{\aleph_{0}}$};
\draw (138,589) node [anchor=north west][inner sep=0.75pt]    {\small cov$(\mathcal{M})$};
\draw (180,482) node [anchor=east] [inner sep=0.75pt]    {\small cov$(\mathcal{N})$};
\draw (255,482) node [anchor=east] [inner sep=0.75pt]    {\small $\mathfrak{r}_{1/2}$};
\draw (347,482) node [anchor=east] [inner sep=0.75pt]    {\small $\mathfrak{r}_{1/2\pm \ \epsilon } \ $};
\draw (440,482) node [anchor=east] [inner sep=0.75pt]    {\small non$(\mathcal{M})$};
\draw (516,510) node [anchor=west] [inner sep=0.75pt]  [font=\small,color={rgb, 255:red, 0; green, 0; blue, 0 }  ,opacity=1 ]  {Dual Hechler model};
\draw (330,455) node [anchor=west] [inner sep=0.75pt]  [font=\small,color={rgb, 255:red, 0; green, 0; blue, 0 }  ,opacity=1 ]  {Hechler model};

\end{tikzpicture}

    \caption{\small Behavior in the Hechler and Dual Hechler model}
    \label{Behvaior in the Hechler model}
\end{figure}
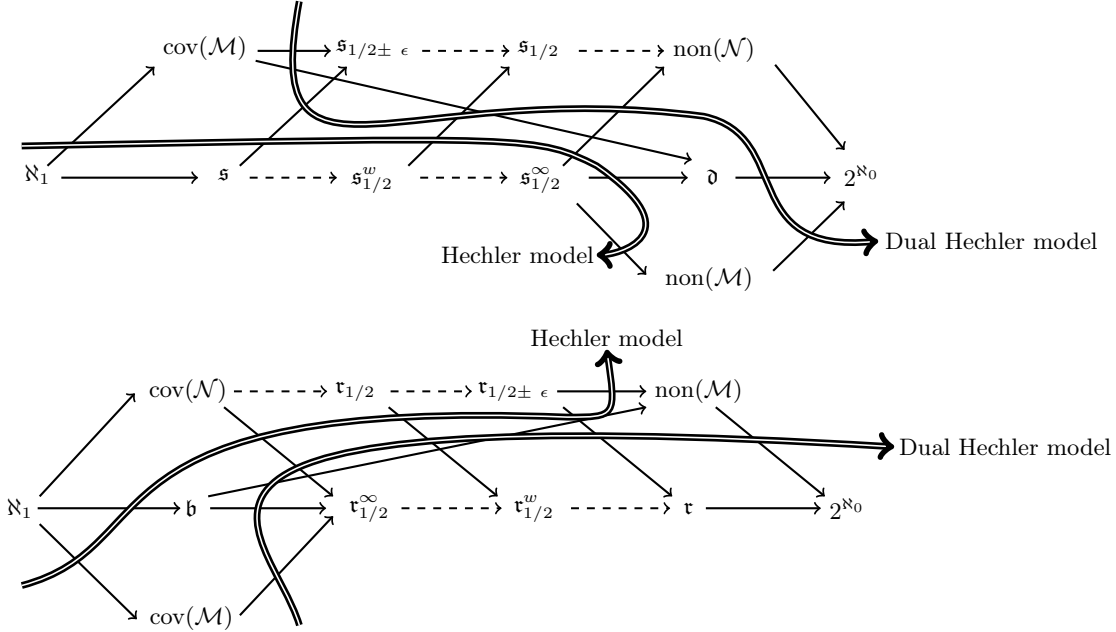

In \cite{Halway2018} was also introduced the following variant of the \textit{independence number} $\mathfrak{i}$, the least cardinality of a maximal independent family. Recall that a family $\mathcal{I}$ of subsets of $\omega$ is called \textit{independent} if for any disjoint finite subfamilies $\mathcal{A}, \mathcal{B} \subseteq \mathcal{I}$, the set
\[
\bigcap_{A \in \mathcal{A}} A \cap \bigcap_{B\in\mathcal{B}} (\omega \setminus B)
\]
is infinite.
\begin{defn}
A set $X\in [\omega]^\omega$ is \textit{moderate} if $\underline{d}(X):=\liminf_{n\to \infty} d_n(X)>0$ as well as $\bar{d}(X):=\limsup_{n\to \infty}d_n(X)<1$.
\end{defn}
\begin{defn}
 We say that a family $\mathcal{I} \subseteq [\omega]^\omega$ is \textit{$*$-independent} if for any set $X \in \mathcal{I}$ we have that $X$ is moderate and for any finite subfamily $\mathcal{F}\subseteq \mathcal{I}$, the following holds
\[
\lim_{n \to \infty} \left(  \frac{d_n(\bigcap_{E\in \mathcal{F}} E )}{\prod_{E\in \mathcal{F}} d_n(E) }   \right)  = 1
\]
We denote the least cardinality of a maximal $*$-independent family by $\mathfrak{i}_{*}$
\end{defn}

In section \ref{independent families}, we show that Con($\mathfrak{i}_{*}<2^{\aleph_0}$) holds in the Cohen model, assuming continuum hypothesis (CH) in the ground model. This solves question six raised in \cite{Halway2018}.

\hspace{2mm}

\textit{Acknowledgments.}  I am grateful to Jörg Brendle for guiding me during this research. His ideas and suggestions contributed to the development of this paper.

\section{$\sigma$-centered forcing}\label{sigma centered forcing}

Let us review the notion of $\sigma$-centered forcing.

\begin{defn}
Let $\mathbb{P}$ be a forcing notion. A subset $P\subseteq \mathbb{P}$ is called \textit{centered} if for all finite $F \subseteq P$ there is $q\in \mathbb{P}$ such that $q\leq p$ for all $p \in F$. $\mathbb{P}$ is \textit{$\sigma$-centered} if there are centered sets $P_n$ such that $\mathbb{P}= \bigcup_{n\in \omega}P_n$.
\end{defn}

Fix $0<\epsilon<1/2$. We prove that $\mathfrak{s}_{1/2\pm \epsilon}$ cannot decrease under iterations of $\sigma$-centered forcing. First, we deal with the single step iteration. 

\begin{lemma}\label{main lemma epsilon families}
Let $\mathbb{P}$ be a $\sigma$-centered forcing notion. Let $\dot{x}$ be a $\mathbb{P}$-name for an element of $[\omega]^\omega$. There are sets $\{X_i\}_{i\in \omega}$ such that for all $A \in [\omega]^\omega$, if for all $i\in \omega$ there are infinitely many $n \in \omega$ such that
\[
\frac{d_n(X_i \cap A)}{d_n(A)} \not\in (1/2 - \epsilon, \epsilon + 1/2)
\]
Then
\[
\Vdash \exists^\infty n \in \omega  \left(\frac{d_n(\dot{x} \cap A)}{d_n(A)} \not\in (1/2 - \epsilon, \epsilon + 1/2) \right) 
\]
\end{lemma}
\begin{proof}
Let $\mathbb{P}= \bigcup_{i\in \omega} \mathcal{P}_i$ where the $\mathcal{P}_i$ are centered. Fix $i \in \omega$. For each $n \in \omega$ there is $y \in \mathcal{P}(\{0, ..., n-1\})$ such that no $p\in \mathcal{P}_i$ forces
\[
\dot{x} \cap n \neq y
\]
By König's lemma, there is $X_i \subset \omega$ such that for all $n \in \omega$ no $p\in \mathcal{P}_i$ forces
\begin{equation}\label{aproximacioneHechler}
 \dot{x} \cap n \neq X_i \cap n   
\end{equation}
Let $A$ be an infinite subset of $\omega$, and suppose that for all $i\in \omega$ there are infinitely many $n \in \omega$ such that
\[
\frac{d_n(X_i \cap A)}{d_n(A)} \not\in (1/2 - \epsilon, \epsilon + 1/2)
\]
Fix $m \in \omega$ and a condition $p \in \mathbb{P}$. There is $\mathcal{P}_i$ such that $p \in \mathcal{P}_i$. Also, there is $n>m$ such that
\[
\frac{d_n(X_i \cap A)}{d_n(A)} \not\in (1/2 - \epsilon, \epsilon + 1/2)
\]
By \ref{aproximacioneHechler} there is $q\leq p$ such that
\[
q \Vdash \dot{x} \cap n = X_i \cap n
\]
Therefore
\[
q \Vdash  \frac{d_n(\dot{x} \cap A)}{d_n(A)} = \frac{d_n(X_i \cap A)}{d_n(A)} \not\in (1/2 - \epsilon, \epsilon + 1/2)
\]
By the previous paragraph the set $D_m=\{ q\in \mathbb{P} \, | \, $ There is $n>m$ such that $ q \Vdash \frac{d_n(\dot{x} \cap A)}{d_n(A)} \not\in (1/2 - \epsilon, \epsilon + 1/2) \}$ is dense. Then 
\[
\Vdash \exists^\infty n \in \omega  \left(\frac{d_n(\dot{x} \cap A)}{d_n(A)} \not\in (1/2 - \epsilon, \epsilon + 1/2) \right) 
\]
\end{proof}

Let us say a partial order has property $(\star)$ if it shares the property exhibited in \ref{main lemma epsilon families}. Note that any forcing notion that has this property does not decrease $\mathfrak{s}_{1/2\pm \epsilon}$. We use a classical argument to prove that $(\star)$ is preserved under finite support iteration (fsi) of ccc forcing notions.   

\begin{lemma}\label{iteration lemma of epsilon families}
Let $\delta$ be a limit ordinal. Let $(\mathbb{P}_\alpha, \dot{\mathbb{Q}}_{\alpha} : \alpha < \delta)$ be an fsi of ccc forcing notions. Assume that $\mathbb{P}_\alpha$ has property $(\star)$ for all $\alpha<\delta$. Then also $\mathbb{P}_\delta$ has property $(\star)$.
\end{lemma}
\begin{proof}
No new reals arise in limit stages of uncountable cofinality so that the lemma vacuously holds if cf$(\delta)>\omega$. Hence assume cf$(\delta)=\omega$. To simplify notation suppose $\delta=\omega$.

Let $\dot{x}$ be a $\mathbb{P}_\omega$-name for an element of $[\omega]^\omega$. Fix $m \in \omega$. Work in $V^{\mathbb{P}_m}$ for the moment. There is a decreasing sequence of conditions $p_k=p_{m,k}$ in the remainder forcing $\mathbb{P}_\omega \setminus \mathbb{P}_m$ such that $p_k$ decides the $k$-th element of $\dot{x}$. Say
\[
p_k \Vdash \text{``$l_k$ is the $k$-th element of $\dot{x}$''}
\]
Let $X_m=\{l_k \,| \, k\in \omega \}$. Work in the ground model $V$. We have a $\mathbb{P}_m$-name $\dot{x}_m$ for $X_m$. Since $\mathbb{P}_m$ has property $(\star)$, we can find $X_{m,i}$, $i\in \omega$ which satisfy what property $(\star)$ said for $\dot{x}_m$. Now, fix $A\in [\omega]^\omega$. We claim that 
if for all $i,m\in \omega$ there are infinitely many $n \in \omega$ such that
\[
\frac{d_n(X_{m,i} \cap A)}{d_n(A)} \not\in (1/2 - \epsilon, \epsilon + 1/2)
\]
then
\[
\Vdash \exists^\infty n \in \omega  \left(\frac{d_n(\dot{x} \cap A)}{d_n(A)} \not\in (1/2 - \epsilon, \epsilon + 1/2) \right) 
\]
To see this, let $p \in \mathbb{P}_\omega$ and $n\in \omega$. Fix $m$ such that $p \in \mathbb{P}_m$ and work in $V^{\mathbb{P}_m}$. We know that there is $j > n$ such that
\[
p_{m,j} \Vdash  \frac{d_{j}(\dot{x} \cap A)}{d_{j}(A)} = \frac{d_{j}(X_m \cap A)}{d_{j}(A)} \not\in (1/2 - \epsilon, \epsilon + 1/2) 
\]
Now work in the ground model $V$. By strengthening $p\in \mathbb{P}_n$, if necessary, we may assume $p$ decides $p_{m,j}$. Then 
\[
q= p^\frown p_{m,j} \Vdash \frac{d_{j}(\dot{x} \cap A)}{d_{j}(A)} \not\in (1/2 - \epsilon, \epsilon + 1/2) 
\]

\end{proof}

Let $\mathbb{P}$ be a forcing notion, and let $\dot{x}$ be a $\mathbb{P}$-name for an element of $[\omega]^\omega$. We say that $\dot{x}$ is a \textit{$\epsilon$-bisecting real} if for all $y\in[\omega]^\omega$ we have that $\Vdash$ ``$\dot{x}$ $\epsilon$-almost bisects  $y$''. Lemma \ref{main lemma epsilon families} and lemma \ref{iteration lemma of epsilon families} immediately imply   

\begin{cor}\label{epsilonnotaddit}
Iterations of $\sigma$-centered forcing do not add $\epsilon$-bisecting reals.
\end{cor}

We proceed to prove the consistency results. We use the Hechler forcing which is a $\sigma$-centered forcing. For the purposes of the next section, we use a slightly different representation
of Hechler forcing.

\begin{defn}
We define Hechler forcing $\mathbb{D}$ as the set of all pairs $(s, \phi)$ such that $s\in \omega^{<\omega}$ and $\phi: \omega^{<\omega} \rightarrow \omega$. The order is given by $(t, \psi) \leq (s, \phi)$ if $t \supseteq s$, $\psi$ dominates $\phi$ everywhere, and  $t(i) \geq \phi(t\upharpoonright i)$ for all $i \in |t| \setminus |s|$.    
\end{defn}

\begin{cor}\label{consistency of d smaller than s1/2}
Con$( \mathfrak{d}< \mathfrak{s}_{1/2\pm \epsilon} )$
\end{cor}
\begin{proof}
Fix $\kappa> \aleph_1$. Let $\mathbb{C}_\kappa$ be a finite support iteration of Cohen forcing of length $\kappa$. Then $V^{\mathbb{C}_\kappa} \models $ cov$(\mathcal{M}) = \mathfrak{s}_{1/2\pm \epsilon} = \kappa$. Denote $M= V^{\mathbb{C}_\kappa}$. Now, let $\mathbb{D}_{\aleph_1}$ be a finite support iteration of Hechler forcing of length $\aleph_1$ over $M$. Then $M^{\mathbb{D}_{\aleph_1}} \models \mathfrak{d} = \aleph_1$ because $\aleph_1$ dominating reals are added. By \ref{iteration lemma of epsilon families} and \ref{main lemma epsilon families}, $\mathbb{D}_{\aleph_1}$ has property $(\star)$, so $M^{\mathbb{D}_{\aleph_1}} \models \mathfrak{s}_{1/2\pm \epsilon}=\kappa$.

\end{proof}

\begin{cor}
Con$( \mathfrak{r}_{1/2\pm \epsilon} <  \mathfrak{b} )$    
\end{cor}
\begin{proof}
Fix $\kappa> \aleph_1$ and assume CH in the ground model. Let $\mathbb{D}_{\kappa}$ be a finite support iteration of Hechler forcing of length $\kappa$. It is known that $V^{\mathbb{D}_\kappa} \models \mathfrak{b}=\kappa$. Besides, the set $V \cap [\omega]^\omega$ remains an $\epsilon$-almost reaping family in $V^{\mathbb{D}_\kappa}$ by \ref{epsilonnotaddit}. Thus $V^{\mathbb{D}_\kappa} \models \mathfrak{r}_{1/2\pm \epsilon}= \aleph_1$. 

\end{proof}

We can obtain a stronger result than \ref{consistency of d smaller than s1/2} using a variation of Mathias forcing.

\begin{defn}
 Let $\mathcal{F}$ be a filter over $\omega$. \textit{Mathias forcing with $\mathcal{F}$}, $\mathbb{M}_{\mathcal{F}}$, consists of all pairs $(s,A)$ where $s$ is a finite subset of $\omega$, and $A \in \mathcal{F}$. The order is given by $(t,B)\leq (s,A)$ if $s\subseteq t$, $B\subseteq A$ and $t \setminus s \subseteq A$. It generically adds a new real $m$ with $m=\bigcup\{s:(s,A)\in G$ for some $A\}$ where $G$ denotes a $\mathbb{M}_{\mathcal{F}}$-generic filter over $V$. It is known that $m$ is a pseudo-intersection of $\mathcal{F}$. Note that $\mathbb{M}_{\mathcal{F}}$ is a $\sigma$-centered forcing. 
\end{defn}

The following proposition is a known result about independent families (see exercises (A12) and (A13) in chapter VIII of \cite{Kunen1980}).

\begin{prop}\label{filter of an independent family}
Let $\mathcal{I}$ be an independent family. There is an filter $\mathcal{F}$ over $\omega$ such that 
\[
\Vdash_{\mathbb{M}_{\mathcal{F}}}  \text{``} \mathcal{I} \cup \{\dot{m}\} \text{ is an independent family''}
\]
where $\dot{m}$ is a $\mathbb{M}_{\mathcal{F}}$-name of the generic real, and for all $x \subseteq \omega$ with $x \not\in \mathcal{I}$ we have
\[
\Vdash_{\mathbb{M}_{\mathcal{F}}}  \text{``} \mathcal{I} \cup \{\dot{m}, x\} \text{ is not an independent family''}
\]
\end{prop}
\begin{proof}
Denote as $\mathcal{I}^\star$ the set of finite Boolean combinations of $\mathcal{I}$. Consider a filter $\mathcal{F}$ over $\omega$ such that
\[
Y \in \mathcal{F} \rightarrow \text{ for all $I\in \mathcal{I}^\star$ we have that  $Y \cap I$  is infinite }
\]
and $\mathcal{F}$ is maximal with  this property. Let $\dot{m}$ be a $\mathbb{M}_{\mathcal{F}}$-name of the generic real. Consider $I \in \mathcal{I}^\star$. For each $n \in \omega$ define
\[
\begin{array}{c}
     B_n := \{ (s,A) \, | \, \text{ there is $m>n$ such that $m \in s \cap I$ } \}  \\ \\ 
     C_n:=  \{ (s,A) \, | \, \text{ there is $m>n$ such that $m \in I$ but $m \not\in s \cup A$} \}
\end{array}
\]
It is easy to check that $B_n$ and $C_n$ are dense sets. Therefore, $\Vdash_{\mathbb{M}_{\mathcal{F}}}  \text{``} I \cap \dot{m} \text{ is infinite''}$ and  $\Vdash_{\mathbb{M}_{\mathcal{F}}}  \text{``} I \cap \omega\setminus\dot{m} \text{ is infinite''}$. Now, consider $x \subseteq \omega$ such that $x \not\in \mathcal{I}$. If $\{x\}\cup \mathcal{I}$ is not an independent family, then there is nothing to do. Thus, suppose that $\{x\}\cup \mathcal{I}$ is an independent family. If $x \in \mathcal{F}$, then $\Vdash_{\mathbb{M}_{\mathcal{F}}} \dot{m} \subseteq^{\star} x$. If $x \not\in \mathcal{F}$, then there is $y \in \mathcal{F}$ and $I \in \mathcal{I}^\star$ such that $y \cap x \cap I$ is finite. Since $\Vdash_{\mathbb{M}_{\mathcal{F}}} \dot{m} \subseteq^{\star} y$, then  $\Vdash_{\mathbb{M}_{\mathcal{F}}} \text{``}\dot{m} \cap x \cap I \text{ is finite''}$.

\end{proof}
\begin{cor}
Con$( \mathfrak{i}< \mathfrak{s}_{1/2\pm \epsilon} )$    
\end{cor}
\begin{proof}
Fix $\kappa> \aleph_1$. Let $\mathbb{C}_\kappa$ be a finite support iteration of Cohen forcing of length $\kappa$. Then $V^{\mathbb{C}_\kappa} \models $ cov$(\mathcal{M}) = \mathfrak{s}_{1/2\pm \epsilon} = \kappa$. Denote $M= V^{\mathbb{C}_\kappa}$. Let $\mathcal{I}_0 \in M$ be a countable independent family. Build in $M$ a iterated forcing construction $\mathbb{P}=(\mathbb{P}_\alpha, \dot{\mathbb{Q}}_\alpha : \alpha < \omega_1)$ such that $\Vdash_\alpha \text{``}\mathcal{I}_0 \cup \{\dot{m}_\beta : \beta< \alpha \} \text{ is an independent family''}$ where $\{\dot{m}_\beta : \beta< \alpha \} $ are the $\mathbb{P}_\alpha$-names of generic reals added by $\mathbb{P}_\alpha$, and $\Vdash_\alpha \text{``}\dot{\mathbb{Q}}_\alpha $ \text{ is a Mathias forcing with $\dot{\mathcal{F}}_\alpha$'' } where $\dot{\mathcal{F}}_\alpha$ is a $\mathbb{P}_\alpha$-name of the filter given by \ref{filter of an independent family} with the 
family $\mathcal{I}_0 \cup \{\dot{m}_\beta : \beta< \alpha \}$. Work in $M^{\mathbb{P}}$. Denote as $\mathcal{I}$ the family $\mathcal{I}_0 \cup \{m_\beta : \beta< \omega_1\}$. By \ref{filter of an independent family}, $\mathcal{I}$ is a maximal independent family. Then, $M^{\mathbb{P}} \models \mathfrak{i}=\aleph_1$. Note that $\mathbb{P}$ has property ($\star$) because of \ref{main lemma epsilon families} and \ref{iteration lemma of epsilon families}. Therefore, $M^{\mathbb{P}} \models \mathfrak{s}_{1/2\pm \epsilon} = \kappa$.

\end{proof}

\section{Hechler forcing} \label{Hechler forcing}

Recall that the ideal $\mathcal{E}$ is the $\sigma$-ideal generated by closed sets of measure zero. By definition, we have that non$(\mathcal{E})$ is smaller than non$(\mathcal{N})$ and non$(\mathcal{M})$. In \cite{Halway2018}, they prove that $\mathfrak{s}_{1/2}^{\infty}$ is also smaller than these two cardinals. It turns out that $\mathfrak{s}_{1/2}^{\infty}\leq$ non$(\mathcal{E})$. 

\begin{prop}
Let $X$ be an infinite subset of $\omega$. Consider the set
\[
B(X)=\{ Y \in [\omega]^\omega \, | \, Y \text{ cofinally bisects } X \}.
\] 
Then $[\omega]^\omega \setminus B(X)$ belongs to $\mathcal{E}$.
\end{prop}
\begin{proof}
Let $\{x_i\}_{i\in \omega}$ be an increasing enumeration of $X$. For each $m\in \omega$ consider the set $ B_{2m} := \{ B \subseteq   \{x_1, ..., x_{2m} \} \, | \, |B|=m \}$, and for each $B\in B_{2m}$ consider the function $f_B: \{x_1, ..., x_{2m} \} \rightarrow \{0,1\}$ where $f_B(x_i)=1$ if and only if $x_i\in B$. Denote as $[f_B]$ the open set generated by $f_B$. Define $A_{2m}:= \bigcup_{B \in B_{2m}} [f_B]$. Note that $Y$ cofinally bisects $X$ if and only if there are infinitely many $m$ such that $Y\in A_{2m}$. Thus
\[
B(X)= \bigcap_{n \in \omega} \bigcup_{m>n} A_{2m}.
\]
Therefore $[\omega]^\omega \setminus B(X)$ is a countable union of closed sets. Furthermore, it is already shown in \cite{Halway2018} that $[\omega]^\omega \setminus B(X)$ is a set of measure zero, so $[\omega]^\omega \setminus B(X)$ is a countable union of closed sets of measure zero. 

\end{proof}

\begin{cor}
 $\mathfrak{s}_{1/2}^{\infty} \leq $ non$(\mathcal{E})$ and cov$(\mathcal{E}) \leq \mathfrak{r}_{1/2}^{\infty}$  
\end{cor}

To prove the consistency of $\mathfrak{s}_{1/2}^{\infty} <$ non$(\mathcal{E})$ and cov$(\mathcal{E}) < \mathfrak{r}_{1/2}^{\infty}$, we work with new objects called good block sequences instead of infinite subsets of $\omega$.

\begin{defn}
We say that a sequence $\mathcal{A}=(A_n : n \in \omega )$ of finite subsets of $\omega$ is a block sequence if $\max(A_n)<\min(A_{n+1})$. We say that a set $B \in [\omega]^\omega$ \textit{splits} $\mathcal{A}$ if both $\{n \in \omega \, | \, A_n \subseteq B \}$ and $\{n \in \omega \, | \, A_n \cap B = \emptyset \}$ are infinite.
\end{defn}

\begin{defn}
Let $\mathcal{A}=(A_n : n \in \omega )$ be a block sequence. We say that $\mathcal{A}$ is a \textit{good block sequence} if there is a recursive function $f \in \omega^\omega$ such that $|A_n|=f(n)$.
\end{defn}

There are two cardinal characteristics associated with the previous notions. 
\[
\begin{array}{ll}
    \mathfrak{fs}&= \min\{ | \mathcal{F} | \, | \,  \mathcal{F} \subseteq [\omega]^\omega \text{ and every block sequence is split by a member of $\mathcal{F}$} \}   \\
    \mathfrak{fs}_{g}&= \min\{ | \mathcal{F} | \, | \,  \mathcal{F} \subseteq [\omega]^\omega \text{ and every good block sequence is split by a member of $\mathcal{F}$} \} 
\end{array}
\]
Kamburelis and W\k{e}glorz showed that $\mathfrak{fs} = \max\{ \mathfrak{b}, \mathfrak{s} \}$ (see \cite{Kamburelis1996-KAMS}). These cardinals also have their dual version.
\[
\begin{array}{ll}
    \mathfrak{fr}&= \min\{ | \mathcal{F} | \, | \,  \mathcal{F} \text{ consists of block sequences and no single $A\in [\omega]^\omega$ splits all members of $\mathcal{F}$} \}   \\
    \mathfrak{fr}_{g}&= \min\{ | \mathcal{F} | \, | \,  \mathcal{F} \text{ consists of good block sequences and no single $A\in [\omega]^\omega$ splits all members of $\mathcal{F}$} \}
\end{array}
\]
In \cite{Brendle1998}, Jörg Brendle showed that $\mathfrak{fr} = \min\{ \mathfrak{d}, \mathfrak{r} \}$. There is a relation with $\mathfrak{s}_{1/2}^{\infty}$ and $\mathfrak{r}_{1/2}^{\infty}$.

\begin{prop}
Let $A, B$ be infinite subsets of $\omega$. Let $\{a_i\}_{i \geq 1}$ be the increasing enumeration of $A$. Define $A_n=\{a_{\sum_{i\leq n}2^i}, ..., a_{(\sum_{i \leq n+1} 2^i ) -1}  \}$ and consider the block sequence $\mathcal{A}=(A_n : n \in \omega )$. Note that $\mathcal{A}$ is a good block sequence because $|A_n|=2^{n+1}$. We have that if $B$ splits $\mathcal{A}$ then $B$ cofinally bisects $A$. 
\end{prop}
\begin{proof}
Assume that $B$ splits $\mathcal{A}$. To see that $B$ cofinally bisects $A$ it suffices to show that both $\{n\in \omega \, | \, \frac{d_n(B\cap A)}{d_n(A)}\geq 1/2 \}$ and $\{n\in \omega \, | \, \frac{d_n(B\cap A)}{d_n(A)}\leq 1/2 \}$ are infinite. Fix $n\in \omega$. There are $m_1,m_2>n$ such that $A_{m_1} \subseteq B$ and $A_{m_2} \cap B = \emptyset$. Letting $K_1:=a_{(\sum_{i \leq m_1+1} 2^i )+1}$ and $K_2:=a_{(\sum_{i \leq m_2+1} 2^i )+1}$ , we have 
\[
\frac{d_{K_1}(B\cap A)}{d_{K_1}(A)} \geq \frac{2^{m_1+1}}{\sum_{i \leq m_1+1} 2^i } \geq \frac{1}{2}  \text{ \, \, \,    and   \, \, \,  } \frac{d_{K_2}(B\cap A)}{d_{K_2}(A)} \leq \frac{\sum_{i \leq m_2} 2^i}{\sum_{i \leq m_2+1} 2^i } \leq \frac{1}{2}
\]

\end{proof}

As an immediate consequence, we get the following

\begin{cor}\label{inequalities fs and fr}
 $\mathfrak{s}_{1/2}^{\infty} \leq \mathfrak{fs}_{g} \leq \mathfrak{fs}$ and  $\mathfrak{fr} \leq \mathfrak{fr}_{g} \leq  \mathfrak{r}_{1/2}^{\infty}$. 
\end{cor}

Now, we are going to prove that $\mathfrak{fs}_{g}=\aleph_1$ in the Hechler model. The proof is analogous to the classical proof that shows that Hechler forcing preserves $\omega$-splitting families (see e.g. \cite[Theorem 3.13]{Brendle2009ForcingAT}).

\begin{lemma} \label{preservation of cofinally spliting families}
Assume $\dot{\mathcal{A}}=(\dot{A}_n : n \in \omega )$ is a $\mathbb{D}$-name of a good block sequence. There are good block sequences $\mathcal{A}_i=(A_j^i : j \in \omega )$, $i\in \omega$, such that whenever $B \in [\omega]^\omega$ splits all $\mathcal{A}_i$, then
\[
\Vdash B  \text{ splits } \dot{\mathcal{A}}
\]

\end{lemma}
\begin{proof}
Let $p\in \mathbb{D}$. There are $r \leq p $ and $f\in \omega^\omega$ a recursive function such that $r \Vdash |\dot{A}_n|=f(n)$. Since we are going to work below $r$, to simplify the notation we can assume
\[
\Vdash |\dot{A}_n|=f(n)
\]
For $s\in \omega^{<\omega}$, $i\in \omega$ and $0<n\leq f(i)$, say that \textit{$s$ favors that $k$ is the $n$-th element of $\dot{A}_i$} if there is no condition with first coordinate $s$ which forces that ``$k$ is not the $n$-th element of $\dot{A}_i$''. Define the \textit{rank} rk$_n^i(s)$ by recursion on the ordinals, as follows:
\begin{itemize}
    \item rk$_n^i(s)=0$ if for some $k \in \omega$, $s$ favors that $k$ is the $n$-th elements of $\dot{A}_i$.
    \item for $\alpha>0$: rk$_n^i(s)=\alpha$ if there is no $\beta< \alpha$ such that rk$_n^i(s)=\beta$ and there are infinitely many $l$ such that rk$_n^i(s^\frown l)<\alpha$.
\end{itemize}

Clearly, rk$_n^i(s)$ must be a countable ordinal or undefined (in which case we write rk$_n^i(s) = \infty$). We show the latter never happens.

\begin{claim}
rk$_n^i(s) < \omega_1$ for all $s\in \omega^{<\omega}$, $i \in \omega$ and $n<f(i)$.
\end{claim}
\begin{proof}
Assume rk$_n^i(s)=\infty$. Notice that for any $t$, if rk$_n^i(t) = \infty$, then rk$_n^i(t^\frown l) = \infty$ for almost all $l \in \omega$. This allows us to recursively construct a function $\phi : \omega^{<\omega} \rightarrow \omega$ such that whenever $t \supset s$ and $t(i) \geq \phi(t\upharpoonright i)$ for all $i \in |t|\setminus|s|$, then  rk$_n^i(t)=\infty$.

Consider the condition $(s, \phi)$. Find $(t, \psi) \leq (s, \phi)$ and $k \in \omega$ such that $(t, \psi)$ forces that $k$ is the $n$-th element of $\dot{A}_i$. Then rk$_n^i(t)=0$. However, by the preceding paragraph, rk$_n^i(t)=\infty$, a contradiction.   

\end{proof}

We continue the proof of the lemma. Consider the function $F(i)= \sum_{j< i} f(j)$. Let $s\in \omega^{<\omega}$. We make a recursive construction to get the block sequences we need. 

\begin{description}
\item[Case 0.] If $s$ is such that rk$_n^i(s)=0$ for infinitely many $(n,i)$, then it is easy to check that rk$_n^i(s)=0$ for all $(n,i)$. For all $m \in \omega$, we can find a finite sequence of finite sets $\mathcal{B}_m=(B_i^m : i<m )$ such that $s$ favors that $\mathcal{B}_m$ are the first $m$ blocks of $\dot{\mathcal{A}}$. Fix $i<\omega$. Choose one block of the collection $\{B_i^m\}_{m > i}$ such that its maximum element coincides with $\min\{ \max(B_i^m) \, | \, m \in \omega \}$, and denote this block as $B_i$. We denote the collection $(B_i : i< \omega)$ as $\mathcal{A}_s$. Note that $\mathcal{A}_s$ is a good block sequence.  

\item[Case 1.]  Let $(n,i) \in \omega \times \omega$ and $s\in \omega^{<\omega}$ such that rk$_n^i(s)=1$. There are infinitely many $l\in \omega$ such that rk$_n^i(s^\frown l)=0$. For each such $l$ we may find $A_l \in [\omega]^{F(i)+n}$ such that $s ^\frown l$ favors that $A_l$ are the first $(F(i)+n)$-th elements of $\bigcup \dot{\mathcal{A}}$. Define:
\[
i_s:= \max \{j \in \omega \,  | \,  \{a_j : a_j \text{ is the $j$-th element of some }A_l \} \text{ is finite } \}
\]
Then there is $A \in [\omega]^{i_s}$ such that the set $\{ l \, | \, A \text{ are the first $i_s$-th elements of } A_l \}$ is infinite. We denote this set as $A_{s,n}^i$. We can construct a block sequence $\mathcal{A}_{s,n}^i$ such that
\begin{itemize}
    \item $\mathcal{A}_{s,n}^i \subseteq [\omega]^{F(i)+n-i_s}$, so it is a good block sequence. Note that we must have $i_s < F(i)+n$ because rk$_n^i(s)>0$.
    \item For each $l\in \omega$ there is $l^\prime > l$ and $D \in \mathcal{A}_{s,n}^i$ such that $s ^\frown l^\prime$ favors that $A_{s,n}^i \cup D$ are the first $(F(i)+n)$-th elements of $\bigcup \dot{\mathcal{A}}$.
\end{itemize}
Note that $s$ favors that $A_{s,n}^i$ are the first $i_s$- elements of $\bigcup \dot{\mathcal{A}}$.

\item[Case $\alpha$.] Let $(n,i) \in \omega \times \omega$ and $s\in \omega^{<\omega}$ such that rk$_n^i(s)=\alpha>1$. There are infinitely many $l\in \omega$ such that $0<$ rk$_n^i(s^\frown l)<\alpha$.  For each $s^\frown l$ there exist a set $A_{s^\frown l,n}^i$ and a natural number $i_{s^\frown l}$ such that $s^\frown l$ favors that $A_{s^\frown l,n}^i$ are the first $i_{s^\frown l}$ elements of $\bigcup \dot{\mathcal{A}}$. Also, there is a block sequence $\mathcal{A}_{s^\frown l,n}^i$. Define:
\[
i_s:= \max \{j \in \omega \,  | \,  \{a_i : a_i \text{ is the $j$-th element of some }A_{s^\frown l,n}^i \} \text{ is finite } \}
\]
Then there is $A \in [\omega]^{i_s}$ such that the set $\{ l \, | \, A \text{ are the first $i_s$-th elements of } A_{s^\frown l,n}^i \}$ is infinite. We denote this set as $A_{s,n}^i$. Also, choose a number $m\leq F(i)+n$ such that the set $\{ l \, | \,  A_{s,n}^i \text{ are the first $i_s$-th elements of } A_{s^\frown l,n}^i \text{ and } m=|A_{s^\frown l,n}^i \setminus A_{s,n}^i| \}$ is infinite. We can construct a block sequence $\mathcal{A}_{s,n}^i$ such that
\begin{itemize}
    \item $\mathcal{A}_{s,n}^i \subseteq [\omega]^{m}$, so it is a good block sequence. Note that $m>0$ since rk$_n^i(s)>0$ and the maximality of $i_s$.
    \item For each $l\in \omega$ there is $l^\prime > l$ and $D \in \mathcal{A}_{s,n}^i$ such that $D = A_{s^\frown l^\prime,n}^i \setminus A_{s,n}^i$.
\end{itemize}
Note that $s$ favors that $A_{s,n}^i$ are the first $i_s$- elements of $\bigcup \dot{\mathcal{A}}$.
   
\end{description}

We claim that if $B$ splits all $\mathcal{A}_s$ and all $\mathcal{A}_{s,n}^i$, then $\Vdash B  \text{ splits } \dot{\mathcal{A}}$. Let $(s, \phi)$ be a condition and let $m\in \omega$. We need to find $(t,\psi) \leq (s, \phi)$ and $m_1,m_2\geq m$ such that $(t,\psi)$ forces that $\dot{A}_{m_1} \subseteq B$ and $\dot{A}_{m_2} \cap B = \emptyset$. Since the construction of $m_1$ and $m_2$ is analogous, it suffices to produce the former.
\begin{description}
    \item[Case 0.] There are infinitely many $(n,i)$ such that rk$_n^i(s)=0$. Since $B$ splits $\mathcal{A}_s$, there is $n>m$ such that $B_n \in \mathcal{A}_s $ and $B_n \subseteq B$. By definition of $\mathcal{A}_s$, $s$ favors that $B_n$ is the $n$-th block of $\dot{\mathcal{A}}$. Therefore, there is $(t,\psi) \leq (s, \phi) $ such that $(t,\psi) \Vdash \dot{A}_n = B_n \subseteq B$.

\end{description}    

Now, suppose that rk$_n^i(s)>0$ for all but finitely many $(n,i)$. Choose $i\geq m$ such that rk$_n^i(s)>0$ for all $0<n\leq f(i)$. There are two cases.

\begin{description}

    \item[Case rk$_{f(i)}^i(s) =1 $.]  Since $B$ splits $\mathcal{A}_{s,f(i)}^i$, we can find $l\geq \phi(s)$ and $D\in \mathcal{A}_{s,f(i)}^i$ such that
    \begin{itemize}
        \item $D \subseteq B$ and $s ^\frown l$ favors that $A_{s,f(i)}^i \cup D$ are the first $(F(i)+f(i))$-th elements of $\bigcup \dot{\mathcal{A}}$.
    \end{itemize}
    Note that $|A_{s,f(i)}^i|=i_s \leq F(i)$ because rk$_n^i(s)>0$ for all $0<n\leq f(i)$, thus $s ^\frown l$ favors that $D$ contains the $i$-th block of $\dot{\mathcal{A}}$. Therefore, there is $(r,\psi) \leq (s, \phi) $ such that 
   \[
   (r,\psi) \Vdash \dot{A}_i \subseteq D \subseteq B
   \]

    \item[Case rk$_{f(i)}^i(s) = \alpha > 1  $.]  Since $B$ splits $\mathcal{A}_{s,f(i)}^i$, we can find $l_1\geq \phi(s)$ and $D_1\in \mathcal{A}_{s,f(i)}^i$ such that
    \begin{itemize}
        \item $1\leq $ rk$_{f(i)}^i(s^\frown l_1) <$ rk$_{f(i)}^i(s)$
        \item $D_1 \subseteq B$ and $D_1 = A_{s^\frown l_1,f(i)}^i \setminus A_{s,f(i)}^i$.
    \end{itemize}
    Repeating this argument, we can extend $s$ to $s^\frown l_1 ^{\frown ... \frown} l_k=t$ and find finite sets $D_1, ..., D_k$ such that
    \begin{itemize}
        \item $t(i)\geq \phi(t\upharpoonright i)$ for all $i \in |t| \setminus |s|$, and rk$_{f(i)}^i(t)=1$.
        \item $D_j \subseteq B$ and $D_j = A_{s^\frown l_1 ^{\frown ... \frown} l_j,f(i)}^i \setminus A_{l_1 ^{\frown ... \frown} l_{j-1},f(i)}^i$
    \end{itemize}
    Finally, since rk$_{f(i)}^i(t)=1$ and $B$ splits $\mathcal{A}_{t,f(i)}^i$, we can find $l\geq \phi(t)$ and $D\in \mathcal{A}_{t,f(i)}^i$ such that
    \begin{itemize}
        \item $D \subseteq B$ and $t ^\frown l$ favors that $A_{t,f(i)}^i \cup D$ are the first $(F(i)+f(i))$-th elements of $\bigcup \dot{\mathcal{A}}$.
    \end{itemize}
   Note that $|A_{s,f(i)}^i|=i_s \leq F(i)$ because rk$_n^i(s)>0$ for all $0<n\leq f(i)$. Also, by construction, $A_{t,f(i)}^i = D_1 \cup ... \cup D_k \cup A_{s,f(i)}^i $, thus $t ^\frown l$ favors that $D_1 \cup ... \cup D_k \cup D$ contains the $i$-th block of $\dot{\mathcal{A}}$. Therefore, there is $(r,\psi) \leq (s, \phi) $ such that 
   \[
   (r,\psi) \Vdash \dot{A}_i \subseteq D_1 \cup ... \cup D_k \cup D \subseteq B
   \]

\end{description}

\end{proof}

Let us say a partial order has property $(\star \star)$ if it shares the property exhibited in \ref{preservation of cofinally spliting families}.  This property is preserved under fsi of ccc forcing notions. 

\begin{lemma}\label{iteration lemma of preservation of cofinally spliting families}
Let $\delta$ be a limit ordinal. Let $(\mathbb{P}_\alpha, \dot{\mathbb{Q}}_{\alpha} : \alpha < \delta)$ be an fsi of ccc forcing notions. Assume that $\mathbb{P}_\alpha$ has property $(\star\star)$ for all $\alpha<\delta$. Then also $\mathbb{P}_\delta$ has property $(\star\star)$.
\end{lemma}
\begin{proof}
No new reals arise in limit stages of uncountable cofinality so that the lemma vacuously holds if cf$(\delta)>\omega$. Hence assume cf$(\delta)=\omega$. To simplify notation suppose $\delta=\omega$.

Let $\dot{\mathcal{A}}=(\dot{A}_n : n \in \omega )$ be a $\mathbb{P}_\omega$-name of a good block sequence. Fix $m \in \omega$. Work in $V^{\mathbb{P}_m}$ for the moment. There is a decreasing sequence of conditions $p_k=p_{m,k}$ in the remainder forcing $\mathbb{P}_\omega \setminus \mathbb{P}_m$ such that $p_k$ decides the $k$-th block of $\dot{\mathcal{A}}$. Say
$p_k \Vdash B_k= \dot{A}_k$. Let $\mathcal{A}_m=(B_n: n \in \omega )$. 

Work in the ground model $V$. We have a $\mathbb{P}_m$-name $\dot{\mathcal{A}}_m$ for $\mathcal{A}_m$. Since $\mathbb{P}_m$ has property $(\star\star)$, there are good block sequences $\mathcal{A}_{m,i}$, $i\in \omega$, such that whenever $B \in [\omega]^\omega$ splits all $\mathcal{A}_{m,i}$, then $\Vdash B  \text{ splits } \dot{\mathcal{A}}_m$.

Now, fix $B\in [\omega]^\omega$. We claim that if $B \in [\omega]^\omega$ splits all $\mathcal{A}_{m,i}$ for all $i,m\in \omega$, then $\Vdash B  \text{ splits } \dot{\mathcal{A}}$.

To see this, let $p \in \mathbb{P}_\omega$ and $n\in \omega$. We need to find $q \leq p$ and $m_1,m_2\geq n$ such that $q$ forces that $\dot{A}_{m_1} \subseteq B$ and $\dot{A}_{m_2} \cap B = \emptyset$. Since the construction of $m_1$ and $m_2$ is analogous, it suffices to produce the former. Fix $m$ such that $p \in \mathbb{P}_m$ and work in $V^{\mathbb{P}_m}$. We know that $B$ splits $\mathcal{A}_m$. Thus there is $k > n$ such that $B_k\subseteq B$ and
\[
p_{m,k} \Vdash  B_k= \dot{A}_k
\]
Now work in the ground model $V$. By strengthening $p\in \mathbb{P}_n$, if necessary, we may assume $p$ decides $p_k$ and $B_k$. Then 
\[
q= p^\frown p_{m,k} \Vdash \dot{A}_k \subseteq B
\]
    
\end{proof}

\begin{cor}\label{values in hechler model 2}
    $V^{\mathbb{D}_\kappa}\models \aleph_1 =\mathfrak{s}_{1/2}^{\infty} =\mathfrak{fs}_{g}  < \mathfrak{fs} = \kappa$ where $\mathbb{D}_\kappa$ is a finite support iteration of Hechler forcing of length $\kappa > \aleph_1$.
\end{cor}
\begin{proof}
    Let $c_\alpha$, $\alpha< \omega_1$ denote the Cohen reals added by the first $\aleph_1$ many Hechler reals. It is easy to see that every countable family of good block sequences in $V^{\mathbb{D}_{\aleph_1}}$ is split by some $c_\alpha$. By the previous lemmas $\mathbb{D}_\kappa$ has property $(\star\star)$, so every good block sequence in $V^{\mathbb{D}_{\kappa}}$ is split by $\{c_\alpha\}_{\alpha<\omega_1}$. Therefore, by \ref{inequalities fs and fr}, we get $V^{\mathbb{D}_\kappa}\models \aleph_1 = \mathfrak{fs}_{g} =\mathfrak{s}_{1/2}^{\infty}$. Besides, since $\mathfrak{fs} = \max\{ \mathfrak{b}, \mathfrak{s} \}$, we have $V^{\mathbb{D}_\kappa}\models \mathfrak{b}=\mathfrak{fs} = \kappa$.
    
\end{proof}

To prove that we can obtain models where $\mathfrak{s}_{1/2}^{\infty} <$ non$(\mathcal{E})$ and cov$(\mathcal{E}) < \mathfrak{r}_{1/2}^{\infty}$ we need the following theorem.

\begin{theorem}\label{Equalities with the ideal E}
 \begin{enumerate}
    \item  If cov$(\mathcal{M}) = \mathfrak{d}$, then cov$(\mathcal{E}) = \max\{$cov$(\mathcal{M})$, cov$(\mathcal{N}) \}$. 
    \item  If non$(\mathcal{M})=\mathfrak{b}$, then non$(\mathcal{E})=\min\{$non$(\mathcal{M})$, non$(\mathcal{N}) \}$.
\end{enumerate}   
\end{theorem}
\begin{proof}
    See \cite[Theorem 2.6.9]{Bartoszynski1995SetTO} 
\end{proof}

\begin{cor}
    Con$( \mathfrak{s}_{1/2}^{\infty} <$ non$(\mathcal{E}) ) $
\end{cor}
\begin{proof}
Let $\kappa>\aleph_1$. By corollary \ref{values in hechler model 2}, $V^{\mathbb{D}_\kappa}\models\mathfrak{s}_{1/2}^{\infty}= \aleph_1$. Furthermore, it is known that $V^{\mathbb{D}_\kappa}\models$ non$(\mathcal{M}) = \mathfrak{b} =$ non$(\mathcal{N}) = \kappa$, then $V^{\mathbb{D}_\kappa}\models$ non$(\mathcal{E})= \kappa$ by theorem \ref{Equalities with the ideal E}.  

\end{proof}

\begin{cor}
  Con$($cov$(\mathcal{E}) < \mathfrak{r}_{1/2}^{\infty} )$
\end{cor}
\begin{proof}
 Fix $\kappa> \aleph_1$. Let $\mathbb{C}_\kappa$ be a finite support iteration of Cohen forcing of length $\kappa$. Because of $\mathfrak{fr}=\min\{\mathfrak{d},\mathfrak{r}\}$ and \ref{inequalities fs and fr}, we have $V^{\mathbb{C}_\kappa} \models \kappa =$ cov$(\mathcal{M}) = \mathfrak{d}=\mathfrak{r} = \mathfrak{fr} = \mathfrak{fr}_g = \mathfrak{r}_{1/2}^\infty$. Denote $M= V^{\mathbb{C}_\kappa}$. Now, let $\mathbb{D}_{\aleph_1}$ be a finite support iteration of Hechler forcing of length $\aleph_1$ over $M$. Then $M^{\mathbb{D}_{\aleph_1}} \models \mathfrak{d} =$ cov$(\mathcal{M}) =$ cov$(\mathcal{N}) = \aleph_1$ because $\aleph_1$ dominating reals are added. Hence $M^{\mathbb{D}_{\aleph_1}} \models \aleph_1 =$ cov$(\mathcal{E})$ by theorem \ref{Equalities with the ideal E}. Also,  $M^{\mathbb{D}_{\aleph_1}} \models \kappa = \mathfrak{fr}_g = \mathfrak{r}_{1/2}^{\infty} $ because $\mathbb{D}_{\aleph_1}$ has property $(\star\star)$. 
 
\end{proof}

\section{$*$-independent families}\label{independent families}

We show that there exists a maximal $*$-independent family of size $\aleph_1$ in the Cohen model. The proof uses similar ideas as \cite[Lemma 3.16]{Halway2018}, but there are more technical details to deal with. First, since any member of an independent family is moderate, we have to construct sets with this property. The following lemmas will help us in achieving this objective. 

\begin{lemma}\label{three times}
If $R,S \subseteq \omega$, $0<\epsilon<1$ and $m\leq l$ are such that
\[
\epsilon < \frac{|R \cap m|}{m}, \frac{|S \cap m|}{m}, \frac{|S \cap l|}{l}  < 1 - \epsilon
\]
then
\[
\frac{\epsilon^2}{2} < \frac{|(R\cap m) \cup (S \cap [m,l))| }{l} < 1- \frac{\epsilon^2}{2}
\]
\end{lemma}
\begin{proof}
First suppose that $l < \frac{2m}{\epsilon}$. Note that
\[
\displaystyle
\begin{array}{ll}
  \frac{|(R \cap m) \cup (S\cap [m,l) )|}{l}   &= \frac{|R \cap m |}{l} + \frac{|S\cap [m,l )|}{l }  \\
     &\leq \frac{m}{l}(1-\epsilon) + \frac{l-m}{l} \\
     &= 1- \epsilon\frac{m}{l} \\
     &<1 - \frac{\epsilon^2}{2}
\end{array}
\]
For the lower bound, we get
\[
\displaystyle
\begin{array}{ll}
  \frac{|(R \cap m) \cup (S\cap [m,l) )|}{l}   &= \frac{|R \cap m |}{l} + \frac{|S\cap [m,l )|}{l }  \\
     &\geq \frac{m}{l}(\epsilon)   \\
     &> \frac{\epsilon^2}{2}
\end{array}
\]
Now, suppose this were false for some $l^{\star} \geq \frac{2m}{\epsilon}$. Then, without loss of generality   
\[
\frac{|(R \cap m) \cup (S\cap [m,l^{\star}) )|}{l^{\star}} \geq 1-\frac{\epsilon^2}{2} .
\]
Since
\[
\frac{|R \cap m |}{m} < 1-\epsilon,
\]
we get
\[
\frac{|S\cap [m,l^{\star})|}{l^{\star}} \geq 1-\frac{\epsilon^2}{2} - \frac{m}{l^{\star}}(1-\epsilon).
\]
But then
\[
\frac{|S \cap m|}{m} > \epsilon
\]
implies
\[
\displaystyle
\begin{array}{ll}
  \displaystyle \frac{|S \cap l^{\star}|}{l^{\star}}    = \frac{|(S \cap m) \cup (S\cap [m,l^{\star}) ) |}{l^{\star}} & \displaystyle > \frac{m}{l^{\star}}(\epsilon) + 1-\frac{\epsilon^2}{2} - \frac{m}{l^{\star}}(1-\epsilon)  \\ 
     &\displaystyle = 1-\frac{\epsilon^2}{2} + \frac{2m}{l^{\star}}(\epsilon) -  \frac{m}{l^{\star}} \\ \\
     &\displaystyle >  1-\frac{\epsilon^2}{2} -  \frac{m}{l^{\star}}   \\
     &\displaystyle \geq 1- \epsilon
\end{array}
\]
which is a contradiction. 

\end{proof}

\begin{lemma}\label{union of intervals}
If $R,S \subseteq \omega$ are disjoint finite sets of sizes $r$ and $s$, respectively, $s=c\cdot r$ for some $c>1$, $q_0,q_1 \in (0,1)$ and $A\subseteq R$, $B\subseteq S$ such that
\[
q_0 \leq \frac{|B|}{|S|}  \leq q_1
\]
then
\[
q_0 - \frac{1}{c} \leq \frac{|A \cup B |}{|R \cup S|}  \leq q_1  + \frac{1}{c}
\]
\end{lemma}
\begin{proof}
Same proof as \cite[Lemma 3.14]{Halway2018}.      
\end{proof}

\begin{cor}\label{Moderate union of intervals}
Let $k\in \omega$ and $0< \epsilon_0 < 1$. Let $Z, W \subseteq \omega$ be such that $Z\subseteq k$ and $\epsilon_0 < \underline{d}(W) \leq \bar{d}(W) < 1-\epsilon_0$. Denote $Z \cup (W\cap [k,l))$ by $Z_l$. Then, there exists $k^\prime \in \omega$ such that for any $l \geq k^\prime$ we have
\[
\epsilon_0 < d_l(Z_l) < 1 - \epsilon_0
\]
\end{cor}
\begin{proof}
Let $\delta>0$ be such that $\epsilon_0 < \delta \leq \underline{d}(W) \leq \bar{d}(W) \leq 1- \delta < 1-\epsilon_0$. First, notice that
\[
\begin{array}{c}
     \displaystyle\limsup_{n \to \infty} \frac{|W \cap n|}{n} = \limsup_{n \to \infty} \frac{|W \cap [k,n)|}{n-k}   \\
     \displaystyle\liminf_{n \to \infty} \frac{|W \cap n|}{n} = \liminf_{n \to \infty} \frac{|W \cap [k,n)|}{n-k}  
\end{array}
\]
then, there is $N>0$ such that for all $l>N$ we have
\[
\delta \leq \frac{|W \cap [k,l)|}{l-k}\leq 1 - \delta
\]
Thus, for any $l>N$, we can apply lemma \ref{union of intervals} to $R:= k$, $S:=[k,l)$, $A:= Z$, $B:=W\cap [k,l)$, $q_0:= \delta$ and $q_1:=1-\delta$ to obtain
\[
\delta - \frac{1}{c} \leq \frac{|Z_l|}{l} \leq 1-\delta + \frac{1}{c}
\]
where $c=(l-k)/k$. Choose $k^\prime$ large enough to guarantee $\epsilon_0 <\delta-\frac{1}{c}$.

\end{proof}

\begin{cor}\label{density union of intervals}
Let $k\in \omega$ and $0< q < 1$. Let $Z, W \subseteq \omega$ be such that $Z\subseteq k$ and $d(W)=q$. Denote $Z \cup (W\cap [k,l))$ by $Z_l$. Then, for all $\delta>0$ there exists $k^\prime \in \omega$ such that for any $l \geq k^\prime$ we have
\[
d_l(Z_l) \in (q -\delta, q + \delta) 
\]    
\end{cor}
\begin{proof}
Same proof as \ref{Moderate union of intervals}.

\end{proof}

Second, to get a maximal $*$-independent family, we make a recursive construction in which we have to extend an $*$-independent family already constructed. The following lemmas will help us with this.

\begin{lemma}\label{joinpieces}
Let $k\in \omega$ and $0<\epsilon<1$. Let $\mathcal{B}=\{B_i\}_{i< n}$ be an $*$-independent family and $Z$ a finite subset of $\omega$ such that for all $E \subseteq n$ we have
\[
\displaystyle  \frac{d_{k}(\bigcap_{i \in E} B_i \cap Z)}{\prod_{i\in E}d_{k}(B_i)d_{k}(Z)} \in \left( 1-\epsilon, 1+ \epsilon \right)
\]    
Also, let $W \subseteq \omega$ and $k^\prime \in \omega$ such that
\begin{itemize}
    \item $\displaystyle  \frac{d_{l}(\bigcap_{i \in E} B_i \cap W)}{\prod_{i\in E}d_{l}(B_i)d_{l}(W)},  \frac{\prod_{i\in E}d_{l}(B_i)d_{l}(W)}{d_{l}(\bigcap_{i \in E} B_i \cap W)} \in \left( 1-\epsilon, 1+ \epsilon \right)$ for all $k \leq l \leq k^\prime$. 
    \item $\displaystyle \frac{|W \cap k|}{|Z \cap k|}, \frac{|Z \cap k|}{|W \cap k|} \in \left( 1-\epsilon, 1+ \epsilon \right)$
\end{itemize}
Then, if we define $Z^\prime:=Z \cup (W\cap [k, k^\prime))$, the following holds
\[
\displaystyle  \frac{d_{l}(\bigcap_{i \in E} B_i \cap Z^\prime)}{\prod_{i\in E}d_{l}(B_i)d_{l}(Z^\prime)} \in \left( 1-15\epsilon, 1+ 15\epsilon \right)
\]
for all $E \subseteq n$  and for all $k \leq l \leq k^\prime$.
\end{lemma}
\begin{proof}
Fix $E \subseteq n$. For any $k \leq l \leq k^\prime$ we have
\[
\begin{array}{ll}
     \displaystyle  \frac{d_{l}(\bigcap_{i \in E} B_i \cap Z^\prime)}{\prod_{i\in E}d_{l}(B_i)d_{l}(Z^\prime)} &=  \displaystyle \frac{1}{\prod_{i\in E}d_{l}(B_i)}\frac{|\bigcap_{i \in E} B_i \cap Z\cap k|+|\bigcap_{i \in E} B_i \cap W \cap [k,l)|}{|Z\cap k|+|W\cap [k,l)|}    \\ \\
     &\leq  \displaystyle\left(1+  \epsilon \right)\left(  \frac{ |W \cap l |}{ |\bigcap_{i \in E} B_i \cap W \cap l|} \right)\left( \frac{|\bigcap_{i \in E} B_i \cap Z \cap k|+|\bigcap_{i \in E} B_i \cap W \cap [k ,l)|}{|Z\cap k|+|W\cap [k,l)|}  \right) \\ \\
     &= \displaystyle\left(1+ \epsilon\right)\left(\frac{|W\cap l|}{ |Z\cap k|+|W\cap [k,l)|} \right) \left( \frac{|\bigcap_{i \in E} B_i \cap Z \cap k|+|\bigcap_{i \in E} B_i \cap W \cap [k,l)|}{|\bigcap_{i \in E} B_i \cap W \cap l|} \right) \\ \\
     &= \displaystyle\left(1+ \epsilon\right)\left(\frac{|W\cap k|+ l^\prime}{ |Z\cap k|+l^\prime} \right) \left( \frac{|\bigcap_{i \in E} B_i \cap Z \cap k|+l^{\prime\prime}}{|\bigcap_{i \in E} B_i \cap W \cap k|+l^{\prime\prime}} \right)     
\end{array}
\]
where $l^\prime=|W\cap [k,l)|$ and $l^{\prime\prime}=|\bigcap_{i \in E} B_i \cap W \cap [k,l)|$. We have 4 cases:
\begin{description}
\item[$|W\cap k|\leq|Z\cap k|$ and $|\bigcap_{i \in E} B_i \cap W \cap k| \geq |\bigcap_{i \in E} B_i \cap Z \cap k|$.] In this case, we get
\[
\left(\frac{|W\cap k|+ l^\prime}{ |Z\cap k|+l^\prime} \right) \left( \frac{|\bigcap_{i \in E} B_i \cap Z \cap k|+l^{\prime\prime}}{|\bigcap_{i \in E} B_i \cap W \cap k|+l^{\prime\prime}} \right)  \leq 1
\]
\item[$|W\cap k|\geq|Z\cap k|$ and $|\bigcap_{i \in E} B_i \cap W \cap k|\leq |\bigcap_{i \in E} B_i \cap Z \cap k|$.] Note that
\begin{equation}\label{equationlemmaunion}
\displaystyle \frac{|W\cap k |}{ |Z \cap k|}  \frac{|\bigcap_{i \in E} B_i \cap Z \cap k |}{|\bigcap_{i \in E} B_i \cap W \cap k |}  =  \frac{\prod_{i\in E}d_{k}(B_i) |W\cap k|}{|\bigcap_{i \in E} B_i \cap W \cap k|  }  \frac{|\bigcap_{i \in E} B_i \cap Z \cap k|}{\prod_{i\in E}d_{k}(B_i) |Z\cap k|} \leq \left(1+ \epsilon \right)^2
\end{equation}
Therefore, in this case, we get
\[
\left(\frac{|W\cap k|+ l^\prime}{ |Z\cap k|+l^\prime} \right) \left( \frac{|\bigcap_{i \in E} B_i \cap Z \cap k|+l^{\prime\prime}}{|\bigcap_{i \in E} B_i \cap W \cap k|+l^{\prime\prime}} \right)  \leq  \left(\frac{|W\cap k|}{ |Z\cap k|} \right) \left( \frac{|\bigcap_{i \in E} B_i \cap Z \cap k|}{|\bigcap_{i \in E} B_i \cap W \cap k|} \right) \leq  \left(1+ \epsilon \right)^2
\]
\item[$|W\cap k|\geq|Z\cap k|$ and $|\bigcap_{i \in E} B_i \cap W \cap k|\geq|\bigcap_{i \in E} B_i \cap Z \cap k|$.] In this case, we have
\[
\left(\frac{|W\cap k|+ l^\prime}{ |Z\cap k|+l^\prime} \right) \left( \frac{|\bigcap_{i \in E} B_i \cap Z \cap k|+l^{\prime\prime}}{|\bigcap_{i \in E} B_i \cap W \cap k|+l^{\prime\prime}} \right)  \leq  \frac{|W\cap k|}{ |Z\cap k|}   \leq  1+ \epsilon 
\]
\item[$|W\cap k|\leq|Z\cap k|$ and $|\bigcap_{i \in E} B_i \cap W \cap k|\leq|\bigcap_{i \in E} B_i \cap Z \cap k|$.]  Note that in \ref{equationlemmaunion} we also get 
\[
\frac{|\bigcap_{i \in E} B_i \cap Z \cap k|}{|\bigcap_{i \in E} B_i \cap W \cap k| } \leq  (1+\epsilon)^2 \frac{|Z \cap k|}{|W \cap k|} \leq (1+ \epsilon)^3
\]
Thus, in this case we have
\[
\left(\frac{|W\cap k|+ l^\prime}{ |Z\cap k|+l^\prime} \right) \left( \frac{|\bigcap_{i \in E} B_i \cap Z \cap k|+l^{\prime\prime}}{|\bigcap_{i \in E} B_i \cap W \cap k|+l^{\prime\prime}} \right)  \leq   \frac{|\bigcap_{i \in E} B_i \cap Z \cap k|}{|\bigcap_{i \in E} B_i \cap W \cap k| }    \leq  (1+ \epsilon)^3
\]
\end{description}

To summarize, we get
\[
\displaystyle  \frac{d_{l}(\bigcap_{i \in E} B_i \cap Z^\prime)}{\prod_{i\in E}d_{l}(B_i)d_{l}(Z^\prime)} \leq (1+\epsilon)^4 \leq 1+ 15\epsilon
\]
In a similar way, we have the lower bound.

\end{proof}

\begin{lemma}\label{stabilizepieces}
Let $k\in \omega$ and $0<\epsilon<1$. Let $\mathcal{B}=\{B_i\}_{i< n}$ be an $*$-independent family and $Z$ a subset of $k$. Let $W$ be a subset of $\omega$ such that $\mathcal{B} \cup \{W\}$ is an $*$-independent family. Denote $Z \cup (W\cap [k, l))$ by $Z_l$. Then, there exists $k^\prime \in \omega$ such that for all $E \subseteq n$ and $l \geq k^\prime$ we have
\[
\displaystyle  \frac{d_{l}(\bigcap_{i \in E} B_i \cap Z_l)}{\prod_{i\in E}d_{l}(B_i)d_{l}(Z_l)} \in \left( 1- \epsilon, 1+ \epsilon  \right)
\]
\end{lemma}

\begin{proof}
First, note that
\[
\displaystyle \lim_{l \to \infty} \frac{d_{l}(\bigcap_{i \in E} B_i \cap W)}{\prod_{i\in E}d_{l}(B_i)d_{l}(W)}  = \displaystyle \lim_{l \to \infty} \frac{ |\bigcap_{i \in E} B_i \cap W \cap [k, l)| }{\prod_{i\in E}d_{l}(B_i) |W \cap [k, l)|}
\]
for any $E\subseteq n$. Therefore, there is $N\in \omega$ such that for all $E\subseteq n$ and $l>N$ we have
\[
\displaystyle  \frac{|\bigcap_{i \in E} B_i \cap W \cap [k,l)|}{\prod_{i\in E}d_{l}(B_i) |W \cap [k, l)|} \in \left( 1-\frac{\epsilon}{2}, 1+ \frac{\epsilon}{2} \right)
\]
Choose some $k^\prime> N$. For the upper bound, we have
\[
\begin{array}{ll}
   \displaystyle  \frac{d_{k^\prime}(\bigcap_{i \in E} B_i \cap Z_{k^\prime})}{\prod_{i\in E}d_{k^\prime}(B_i)d_{k^\prime}(Z_{k^\prime})} & \leq \displaystyle\frac{ k + |\bigcap_{i \in E} B_i \cap W \cap [k,k^\prime)| }{\prod_{i\in E}d_{k^\prime}(B_i) |W \cap [k, k^\prime)|}    \\ \\
     & \leq \displaystyle \frac{k(1+ \frac{\epsilon}{2})}{|\bigcap_{i \in E} B_i \cap W \cap [k, k^\prime)|}   + 1+ \frac{\epsilon}{2}
\end{array}
\]
For the lower bound, we have
\[
\begin{array}{ll}
   \displaystyle  \frac{d_{k^\prime}(\bigcap_{i \in E} B_i \cap Z_{k^\prime})}{\prod_{i\in E}d_{k^\prime}(B_i)d_{k^\prime}(Z_{k^\prime})} & \geq \displaystyle\frac{ |\bigcap_{i \in E} B_i \cap W \cap [k, k^\prime)| }{\prod_{i\in E}d_{k^\prime}(B_i) (k+|W \cap [k, k^\prime)|)}    \\ \\
     & \geq  \displaystyle\frac{ (1- \frac{\epsilon}{2}) \prod_{i\in E}d_{k^\prime}(B_i) ||W \cap [k, k^\prime)| }{\prod_{i\in E}d_{k^\prime}(B_i) (k+|W \cap [k, k^\prime)|)} \\ \\
     &= \displaystyle \left(1-\frac{\epsilon}{2}\right) \frac{1}{1+\frac{k}{|W \cap [k, k^\prime)|}}
\end{array}
\]
Since the collection of subsets of $  \{0, ..., n-1\}$ is finite, we can choose $k^\prime$ sufficiently large as to guarantee $(1-\frac{\epsilon}{2}) \frac{1}{1+\frac{k}{|W \cap [k, k^\prime)|}}\geq 1-\epsilon$  and $\frac{k(1+ \epsilon/2)}{|\bigcap_{i \in E} B_i \cap W \cap [k, k^\prime)|} \leq \frac{\epsilon}{2}$ for all $E \subseteq n$.

\end{proof}

Finally, the following lemmas will help us to construct a finite set that satisfies the conditions of all aforementioned lemmas.

\begin{lemma}\label{lemmaabindependetfamilies}
Let $\mathcal{B}=\{B_i\}_{i< n}$ be an $*$-independent family. Then, for every rational number $0<q<1$ there exists an infinite set $X_q$ such that
\begin{itemize}
    \item $d(X_q)=q$
    \item $\mathcal{B} \cup \{X_q\}$ is an $*$-independent family.
\end{itemize}
\end{lemma}
\begin{proof}
First, for each function $f:n \rightarrow \{-1,1\}$ we define
\[
\mathcal{B}_f := \bigcap_{i<n} B_i^{f(i)}
\]
where $B_i^1:=B_i$ and $B_i^{-1}:=\omega \setminus B_i$. Note that the collection $\{ \mathcal{B}_f \, | \, f \in \{-1,1\}^n \}$ is a partition of $\omega$. Fix a function $f \in \{-1,1\}^n$. Let $\{d_j\}_{j\in \omega}$ be an increasing enumeration  of $\mathcal{B}_f$. Let $\frac{a}{b}$ be the irreducible fraction of $q$, and let $A$ be a set containing exactly $a$ residue classes modulo $b$. Define 
\[
X_f := \{ d_j \, | \, j \in \bigcup A \}
\]
It is easy to show that $\displaystyle \lim_{l \to \infty} \frac{|X_f \cap l|}{|\mathcal{B}_f \cap l |} = \frac{a}{b}$. Now, define 
\[
X_q:= \bigcup_{ f \in \{-1,1\}^n } X_f
\]
Let us see that $X_q$ satisfies the conditions of the lemma. It suffices to show that for any $E \subseteq n$ we have
\[
\lim_{l \to \infty}\frac{|\bigcap_{i \in E} B_i \cap X_q \cap l |}{|\bigcap_{i \in E} B_i \cap l |} = \frac{a}{b}
\]
Let $E$ be a subset of $n$. Consider the set $G:=\{  f \in \{-1,1\}^n \,| \, f(i)=1 \text{ for all $i \in E$ } \} $. Note that
\[
\bigcap_{i \in E} B_i = \bigcup_{f \in G} \mathcal{B}_f
\]
Also, given $\epsilon>0$ there exists $N\in \omega$ such that for all $l>N$ we have
\[
\frac{|X_f \cap l|}{|\mathcal{B}_f \cap l |} \in \left( \frac{a}{b}- \epsilon, \frac{a}{b}+\epsilon \right)
\]
for all $f\in G$. Thus, for $l>N$ we get
\[
\begin{array}{ll}
   \displaystyle\frac{|\bigcap_{i \in E} B_i \cap X_q \cap l |}{|\bigcap_{i \in E} B_i \cap l |}  &\displaystyle= \sum_{f \in G} \frac{|\mathcal{B}_f \cap X_q \cap l|}{ |\bigcap_{i \in E} B_i \cap l |}  \\ \\
     &\displaystyle= \sum_{f \in G} \frac{|X_f \cap l|}{|\bigcap_{i \in E} B_i \cap l |} = \sum_{f \in G} \frac{|X_f \cap l|}{|\mathcal{B}_f \cap l |}\frac{|\mathcal{B}_f \cap l |}{|\bigcap_{i \in E} B_i \cap l |} \\ \\
     &\displaystyle\leq \left( \frac{a}{b} + \epsilon\right)\sum_{f \in G}\frac{|\mathcal{B}_f \cap l |}{|\bigcap_{i \in E} B_i \cap l |} = \left( \frac{a}{b} + \epsilon\right)
\end{array}    
\]
In a similar way we get $\left( \frac{a}{b} - \epsilon\right) \leq \frac{|\bigcap_{i \in E} B_i \cap X_q \cap l |}{|\bigcap_{i \in E} B_i \cap l |}$. 

\end{proof}

\begin{lemma}\label{equalizessizes}
Let $k\in \omega$, $\epsilon>0$ and $\epsilon_0>0$. Let $\mathcal{B}=\{B_i\}_{i< n}$ be an $*$-independent family and $Z$ a subset of $k$ such that for all $E \subseteq n$ we have
\[
\displaystyle  \frac{d_{k}(\bigcap_{i \in E} B_i \cap Z)}{\prod_{i\in E}d_{k}(B_i)d_{k}(Z)} \in \left( 1-\epsilon, 1+ \epsilon \right)
\]
Besides, for every $m \geq k$ there is a finite set $Z_m \subseteq m$ such that
\begin{itemize}
    \item $Z_m \cap k = Z$
    \item $\epsilon_0 \leq \frac{|Z_m \cap l|}{l} \leq 1- \epsilon_0 $ for all $k \leq l \leq m$.
    \item $\displaystyle  \frac{d_{l}(\bigcap_{i \in E} B_i \cap Z_m)}{\prod_{i\in E}d_{l}(B_i)d_{l}(Z_m)} \in \left( 1-\epsilon, 1+ \epsilon \right)$ for all $k \leq l \leq m$.
\end{itemize}
Let $W$ be an infinite subset of $\omega$ such that $\epsilon_0 \leq \frac{|W \cap l|}{l} \leq 1- \epsilon_0$ for all $l>k$. Then, there exists $k^\prime> k$ and a finite set $Z^\prime \subseteq k^\prime$ such that
\begin{itemize}
    \item $Z^\prime \cap k = Z$.
    \item $\displaystyle \frac{\epsilon_0^2}{2} \leq \frac{|Z^\prime \cap l|}{l} \leq 1- \frac{\epsilon_0^2}{2} $ for all $k \leq l < k^\prime$.
    \item $\displaystyle  \frac{d_{l}(\bigcap_{i \in E} B_i \cap Z^\prime)}{\prod_{i\in E}d_{l}(B_i)d_{l}(Z^\prime)} \in \left( 1-15\epsilon, 1+ 15\epsilon \right)$ for all $k \leq l < k^\prime$ and any $E \subseteq n$.
    \item $\displaystyle \epsilon_0 \leq \frac{|Z^\prime \cap k^\prime|}{k^\prime} \leq 1- \epsilon_0 $ 
    \item $\displaystyle  \frac{d_{k^\prime}(\bigcap_{i \in E} B_i \cap Z^\prime)}{\prod_{i\in E}d_{k^\prime}(B_i)d_{k^\prime}(Z^\prime)} \in \left( 1- \epsilon, 1+ \epsilon \right)$ for any $E \subseteq n$.
    \item $\displaystyle \frac{|Z^\prime \cap k^\prime|}{|W \cap k^\prime|}, \frac{|W \cap k^\prime|}{|Z^\prime \cap k^\prime|} \in (1 -\epsilon, 1 + \epsilon)$
\end{itemize}
\end{lemma}
\begin{proof}
First, note that there exists some rational $\delta>0$ such that for any $q$ rational number such that $\epsilon_0 \leq q \leq 1- \epsilon_0$ we have
\[
1-\epsilon<\frac{q-\delta}{q+\delta},\frac{q+\delta}{q-\delta} < 1 + \epsilon
\]
In fact, it is easy to prove that any $\delta < \frac{\epsilon_0 \epsilon}{2+\epsilon}$ satisfies the previous inequality. Now, consider a finite collection of intervals with rational center $\mathcal{I}_0= \{( q_j - \delta, q_j + \delta) \}_{j < J}$ such that $\epsilon_0 < q_j < 1-\epsilon_0$ and $[\epsilon_0,1-\epsilon_0] \subseteq \bigcup \mathcal{I}_0 $. Let $X_{q_j}$ be the sets which are given by lemma \ref{lemmaabindependetfamilies}. Let $k_0 \in \omega$ such that for any $E \subseteq n$, $ j < J$ and for all $l \geq k_0$ we have
\begin{itemize}
    \item $ \displaystyle  \frac{d_{l}(\bigcap_{i \in E} B_i \cap X_{q_j})}{\prod_{i\in E}d_{l}(B_i)d_{l}(X_{q_j})},  \frac{\prod_{i\in E}d_{l}(B_i)d_{l}(X_{q_j})}{d_{l}(\bigcap_{i \in E} B_i \cap X_{q_j})} \in \left( 1- \epsilon, 1+ \epsilon \right) $
    \item $\displaystyle \frac{|X_{q_j}\cap l|}{l} \in \left( q_j - \delta, q_j + \delta \right) $
    \item $\displaystyle \epsilon_0 < \frac{|X_{q_j}\cap l|}{l} < 1-\epsilon_0 $
\end{itemize}
Besides, there is a finite set $Z_{k_0}\subseteq k_0$ such that
\begin{itemize}
    \item $Z_{k_0} \cap k = Z$
    \item $\epsilon_0 \leq \frac{|Z_{k_0} \cap l|}{l} \leq 1- \epsilon_0 $ for all $k \leq l \leq k_0$.
    \item $\displaystyle  \frac{d_{l}(\bigcap_{i \in E} B_i \cap Z_{k_0} )}{\prod_{i\in E}d_{l}(B_i)d_{l}(Z_{k_0} )} \in \left( 1-\epsilon, 1+ \epsilon \right)$ for all $k \leq l \leq k_0 $.
\end{itemize}
Also, there is some $j< J$ such that $\frac{|Z_{k_0} \cap k_0 |}{k_0} \in \left(q_j-\delta, q_j + \delta \right)$. Denote this rational as $r_1$. Then, we get
\[
1 - \epsilon \leq  \frac{r_1-\delta}{r_1+\delta} \leq \frac{|X_{r_1}\cap k_0 |}{ |Z_{k_0} \cap k_0 | } , \frac{|Z_{k_0} \cap k_0 |}{|X_{r_1}\cap k_0 |} \leq  \frac{r_1+\delta}{r_1-\delta} \leq 1 + \epsilon
\]
Denote $Z_{k_0} \cup (X_{r_1} \cap [k_0, l))$ by $Z_l^1$. Therefore, we can apply lemma \ref{three times}, corollary \ref{Moderate union of intervals}, corollary \ref{density union of intervals}, lemma \ref{joinpieces} and \ref{stabilizepieces} with $Z_{k_0}$ and $X_{r_1}$ to find $k_1$ such that 
\begin{itemize}
    \item $\displaystyle  \frac{d_{l}(\bigcap_{i \in E} B_i \cap Z_{l}^1)}{\prod_{i\in E}d_{l}(B_i)d_{l}(Z_{l}^1)} \in \left( 1-15\epsilon, 1+ 15\epsilon \right)$ for all $k_0 \leq l < k_1$.
    \item $\displaystyle \frac{\epsilon_0^2}{2} \leq \frac{|Z_l^1 \cap l|}{l} \leq 1- \frac{\epsilon_0^2}{2} $ for all $k \leq l < k_1$
    \item $\displaystyle  \frac{d_{l}(\bigcap_{i \in E} B_i \cap Z_{l}^1)}{\prod_{i\in E}d_{l}(B_i)d_{l}(Z_{l}^1)} \in \left( 1- \epsilon, 1+ \epsilon \right)$  for all $l \geq k_1$
    \item $\displaystyle \epsilon_0 \leq \frac{|Z_l^1 \cap l|}{l} \leq 1- \epsilon_0 $ for all $l \geq k_1$
    \item $\displaystyle \frac{|Z_{l}^1 \cap l|}{l} \in \left(r_1 -\delta, r_1 + \delta \right)$ for all $l \geq k_1$
\end{itemize}
If there is $l \geq k_1$ such that $\frac{|W \cap l|}{l}\in \left(r_1 -\delta, r_1 + \delta \right)$, then we have the result with $Z^\prime:= Z_l^1$ and $k^\prime=l$. Otherwise, there is $k_1^\prime>k_1$ such that 
\[
\frac{|W \cap l|}{l} < r_1 -\delta \text{ or }  \frac{|W \cap l|}{l} > r_1 +\delta
\]
for all $l\geq k_1^\prime$. Without loss of generality, we can assume that $\frac{|W \cap l|}{l} < r_1 -\delta$ for all $l>k_1^\prime$. Now, define $r_2:= r_1  - \frac{\delta}{2}$. Note that $r_2 >  \epsilon_0$ since $ \frac{|W \cap k_1^\prime|}{k_1^\prime} \geq \epsilon_0$. Consider the set $X_{r_2}$ given by the lemma \ref{lemmaabindependetfamilies}. Since there is $\delta^\prime$ such that $(r_1- \delta^\prime, r_1 + \delta^\prime) \subseteq \left(r_2 -\delta, r_2 + \delta \right)$, by corollary \ref{density union of intervals} there is some $k_1^{\prime\prime} \geq k_1^\prime$ such that for all $l \geq k_1^{\prime\prime}$
\[
\frac{|Z_{l}^1 \cap l|}{l} \in (r_1- \delta^\prime, r_1 + \delta^\prime) \subseteq \left(r_2 -\delta, r_2 + \delta \right)
\]
Besides, if $k_1^{\prime\prime}$ is large enough, we can ask that for any $E \subseteq n$ and for all $l \geq k_1^{\prime\prime}$ 
\begin{itemize}
    \item $ \displaystyle  \frac{d_{l}(\bigcap_{i \in E} B_i \cap X_{r_2})}{\prod_{i\in E}d_{l}(B_i)d_{l}(X_{r_2})},  \frac{\prod_{i\in E}d_{l}(B_i)d_{l}(X_{r_2})}{d_{l}(\bigcap_{i \in E} B_i \cap X_{r_2})} \in \left( 1- \epsilon, 1+ \epsilon \right) $
    \item $\displaystyle \frac{|X_{r_2}\cap l|}{l} \in \left( r_2 - \delta, r_2 + \delta \right) $
    \item $\displaystyle \epsilon_0 < \frac{|X_{r_2}\cap l|}{l} < 1-\epsilon_0 $ 
\end{itemize}
Then, we get
\[
1 - \epsilon \leq  \frac{r_2-\delta}{r_2+\delta} \leq \frac{|X_{r_2}\cap k_1^{\prime\prime} |}{ |Z_{k_1^{\prime\prime}}^1 \cap k_1^{\prime\prime} | } , \frac{|Z_{k_1^{\prime\prime}}^1 \cap k_1^{\prime\prime} |}{|X_{r_2}\cap k_1^{\prime\prime} |} \leq  \frac{r_2+\delta}{r_2-\delta} \leq 1 + \epsilon
\]
Denote $Z_{k_1^{\prime\prime}}^1 \cup (X_{r_2} \cap [k_1^{\prime\prime}, l))$ by $Z_l^2$. Therefore, we can apply again lemma \ref{three times}, corollary \ref{Moderate union of intervals}, corollary \ref{density union of intervals}, lemma \ref{joinpieces} and \ref{stabilizepieces} with $Z_{k_1^{\prime\prime}}^1 $ and $X_{r_2}$ to find $k_2$ such that 
\begin{itemize}
    \item $\displaystyle  \frac{d_{l}(\bigcap_{i \in E} B_i \cap Z_{l}^2)}{\prod_{i\in E}d_{l}(B_i)d_{l}(Z_{l}^2)} \in \left( 1-15\epsilon, 1+ 15\epsilon \right)$ for all $k_1^{\prime\prime} \leq l < k_2$.
    \item $\displaystyle \frac{\epsilon_0^2}{2} \leq \frac{|Z_l^2 \cap l|}{l} \leq 1- \frac{\epsilon_0^2}{2} $ for all $k_1^{\prime\prime} \leq l < k_2$
    \item $\displaystyle  \frac{d_{l}(\bigcap_{i \in E} B_i \cap Z_{l}^2)}{\prod_{i\in E}d_{l}(B_i)d_{l}(Z_{l}^2)} \in \left( 1- \epsilon, 1+ \epsilon \right)$  for all $l \geq k_2$
    \item $\displaystyle \epsilon_0 \leq \frac{|Z_l^2 \cap l|}{l} \leq 1- \epsilon_0 $ for all $l \geq k_2$
    \item $\displaystyle \frac{|Z_{l}^2 \cap l|}{l} \in \left(r_2 -\delta, r_2 + \delta \right)$ for all $l \geq k_2$
\end{itemize}
If there is $l \geq k_2$ such that $\frac{|W \cap l|}{l}\in \left(r_2 -\delta, r_2 + \delta \right)$, then we have the result with $Z^\prime:= Z_l^2$ and $k^\prime=l$. Otherwise, $\frac{|W \cap l|}{l} < r_2 -\delta $ for all $l \geq k_2$. Then, define $r_3= r_2 - \frac{\delta}{2}$ and repeat the previous step. This construction has an ending because $\epsilon_0 \leq \frac{|W \cap l|}{l} \leq 1- \epsilon_0$ for all $l>k$.

\end{proof}

Now, we are ready to begin the proof.

\begin{theorem}
Assume CH in the ground model and let $\kappa\geq\aleph_2$ with $\kappa=\kappa^{\aleph_0}$. Let $\mathbb{C}$ be the Cohen forcing, then $V^{\mathbb{C}_{\kappa}} \models$ ``There is a maximal $*$-independent family of size $\aleph_1$ ''.
\end{theorem}
\begin{proof}

It suffices to consider what happens when forcing with just $\mathbb{C}=$ Fn$(2,\omega)$. Let $\mathcal{A}_0=\{A_n \subseteq [\omega]^\omega \, | \, n \in \omega \}$ be an $*$-independent family. Fix an enumeration $\{(p_\alpha, \dot{X}_{\alpha}) \, | \, \omega \leq \alpha \leq \omega_1 \}$ of all pairs $(p, \dot{X})$ such that $p\in \mathbb{C}$ and $\dot{X}$ is a nice name for a subset of $\omega$. Note that since $V\models$ CH, there are just $\aleph_1$ many nice names for subsets of $\omega$ in $V$.

We now construct a maximal $*$-independent family $\mathcal{A}$ from $\mathcal{A}_0$ as follows. Let $\omega \leq \alpha < \omega_1$ and assume we have already defined set $A_\beta$ for all $\beta<\alpha$. Below, we will construct $A_\alpha \subseteq \omega$ such that the following two properties hold:
\begin{enumerate}[(i)]
    \item The family $\{A_\beta \, | \, \beta \leq \alpha \}$ is $*$-independent.
    \item  If there is a rational number $\epsilon_0>0$ and $N\in \omega$ in the ground model such that 
    \[
    \text{$p_\alpha \Vdash $``$\epsilon_0 < d_l(\dot{X}_\alpha) < 1-\epsilon_0$ for all $l\geq N$, and  $\{A_\beta \, | \, \beta < \alpha \} \cup \{ \dot{X}_\alpha \}$ is $*$-independent'',}
    \]
    then there exists $\epsilon^\prime < \epsilon_0$ and a sequence of intervals of $\omega$, $\mathcal{I}_\alpha=\{ I_j \}_{j \in \omega}$, with $\max I_j < \min I_{j+1}$, such that 
    \begin{enumerate}
        \item $|I_n \cap A_\alpha| > (1-\epsilon^\prime)| (\max I_n +1) \cap A_\alpha|$.
        \item For all $m\in \omega$ the set 
    \[
    D^\alpha_{m}:=\{ q \in \mathbb{C} \, | \, \exists n \geq m : q \Vdash A_\alpha \cap I_n = \dot{X}_{\alpha} \cap I_n \}     
    \]
    is dense below $p_\alpha$.
    \end{enumerate}

\end{enumerate}

We first show that the family $\mathcal{A}=\{A_\beta \, | \, \beta < \omega_1 \}$ constructed this way is a maximal $*$-independent family in $V^{\mathbb{C}}$. Clearly, $\mathcal{A}$ is an $*$-independent family. Suppose it were not maximal. Then, there is a condition $p$ and a nice name $\dot{X}$ for a subset of $\omega$ such that
\[
p \Vdash \text{``$\mathcal{A} \cup \{\dot{X}\}$ is $*$-independent'' }
\]
By strengthening $p$, we can assume that there is a rational number $\epsilon_0$ and $N\in \omega$ such that 
\[
p \Vdash  \epsilon_0 < d_l(\dot{X}) < 1-\epsilon_0
\]
for $l \geq N$. This is possible since $\Vdash$ ``$\dot{X}$ is moderate''. Let $\alpha$ be such that $(p,\dot{X})=(p_\alpha, \dot{X}_\alpha)$, so 
\[
p_\alpha \Vdash \text{``$ \{ A_\alpha, \dot{X}_\alpha\}$ is $*$-independent'' }
\]
By property (ii), there exist $\epsilon^\prime < \epsilon_0$ and a sequence of intervals of $\omega$, $\mathcal{I}_\alpha=\{ I_n \}_{n \in \omega}$, such that $|I_n \cap A_\alpha| > (1-\epsilon^\prime)| (\max I_n +1) \cap A_\alpha|$. Let $\epsilon= \frac{1-\epsilon^\prime}{1-\epsilon_0} - 1$. Find a condition $q\leq p_\alpha$ and $M \in \omega$ such that for all $n\geq M$  
\begin{equation}\label{contradiccion}
q\Vdash \frac{d_n(A_\alpha \cap \dot{X}_\alpha) }{d_n(A_\alpha) \cdot d_n(\dot{X}_\alpha) } \in ( 1- \epsilon, 1+ \epsilon )    
\end{equation}
Consider an interval $I_m$ such that $\min I_m > N, M$. Now by the density of $D^\alpha_{m}$ we can find $r \leq q$ and some $n\geq m$ such that $r \Vdash A_\alpha \cap I_n = \dot{X}_{\alpha} \cap I_n$. Denote $j$ as $\max I_n +1$. We have 
\[
\begin{array}{ll}
  r \Vdash \frac{d_{j}(A_\alpha \cap \dot{X}_\alpha) }{ d_{j}(\dot{X}_\alpha) \cdot d_{j}(A_\alpha) } &= \frac{|A_\alpha \cap \dot{X}_\alpha \cap j|}{ \frac{|\dot{X}_\alpha \cap j | }{j} \cdot |A_\alpha \cap j  | }  \\
     &\geq \frac{|A_\alpha \cap \dot{X}_\alpha \cap j|}{  (1-\epsilon_0) \cdot |A_\alpha \cap j  | }  \\
     & \geq  \frac{|A_\alpha \cap \dot{X}_\alpha \cap I_n|}{  (1-\epsilon_0) \cdot |A_\alpha \cap j  | } = \frac{|A_\alpha \cap I_n|}{  (1-\epsilon_0) \cdot |A_\alpha \cap j  | } > \frac{1-\epsilon^\prime}{1- \epsilon_0} = 1 + \epsilon
\end{array}
\]
which contradicts Eq.( \ref{contradiccion}).

\hspace{2mm}

Now, we are going to show that we can construct such an $A_\alpha$ satisfying the properties (i) and (ii) for $\omega \leq \alpha < \omega_1$. We only have to consider those $\alpha$ such that $\dot{X}_\alpha$ satisfies the assumption in property (ii), since finding an $A_\alpha$ with property (i) is straightforward. Therefore, assume that there is a rational number $\epsilon_0>0$ and $N\in \omega$ in the ground model such that 
    \[
    \text{$p_\alpha \Vdash $``$\epsilon_0 < d_l(\dot{X}_\alpha) < 1-\epsilon_0$ for all $l\geq N$, and  $\{A_\beta \, | \, \beta < \alpha \} \cup \{ \dot{X}_\alpha \}$ is $*$-independent'',}
    \]
and consider $\epsilon^\prime < \epsilon_0$. Let $\{ B_n \, | \, n<\omega\}$  be an enumeration of $\{A_\beta \, | \, \beta < \alpha \}$. We further pick some strictly decreasing sequence of real numbers $\{\delta_n\}_{n\in\omega}$ with $\delta_0<1$ and $\lim_{n \to \infty} \delta_n =0$, and let $\{q_n \, | \, n \in \omega\}$ be some sequence enumerating all conditions below $p_\alpha$ infinitely often. By induction on $n<\omega$, we will construct conditions $r_n \leq q_n^\prime \leq q_n$, a strictly increasing of natural numbers $\{k_n\}_{n \in \omega}$ and initial segments $Z_n=A_\alpha \cap k_n$ of $A_\alpha$ such that the following statements will hold
\begin{multicols}{2}
\begin{description}

  \item[$*$-Independence:] 
  
  \item[C1.]  $\displaystyle  \frac{d_{k_n}(\bigcap_{i \in E} B_i \cap Z_n)}{\prod_{i\in E}d_{k_n}(B_i)d_{k_n}(Z_n)} \in \left( 1-\frac{\delta_n}{15}, 1+ \frac{\delta_n}{15}  \right)$

  for all $E \subseteq \{0, ..., n \}$.

  \item[C2.] $\displaystyle q_n^\prime \Vdash \displaystyle  \frac{d_{l}(\bigcap_{i \in E} B_i \cap \dot{X}_\alpha)}{\prod_{i\in E}d_{l}(B_i)d_{l}(\dot{X}_\alpha)} \in \left( 1-\frac{\delta_n}{15}, 1+ \frac{\delta_n}{15}  \right)$  
  
   $\displaystyle q_n^\prime \Vdash \displaystyle  \frac{\prod_{i\in E}d_{l}(B_i)d_{l}(\dot{X}_\alpha)}{d_{l}(\bigcap_{i \in E} B_i \cap \dot{X}_\alpha)} \in \left( 1-\frac{\delta_n}{15}, 1+ \frac{\delta_n}{15}  \right)$ 
  
  for all $l\geq k_n$ and $E \subseteq \{0, ..., n \}$.

  \item[C3.] $\displaystyle  \frac{d_{l}(\bigcap_{i \in E} B_i \cap Z_{n+1})}{\prod_{i\in E}d_{l}(B_i)d_{l}(Z_{n+1})} \in \left( 1-\delta_n, 1+ \delta_n \right)$ 
  
  for all $k_n \leq l \leq k_{n+1}$ and $E \subseteq \{0, ..., n \}$. 
\end{description}

\columnbreak

\begin{description}

  \item[Moderate:] 
  
  \item[C4.]  $ \epsilon^\prime \leq \displaystyle d_{k_n}(Z_n) \leq 1-\epsilon^\prime $.  

  \item[C5.] $ \frac{(\epsilon^\prime)^2}{2}  \leq \displaystyle d_{l}(Z_{n+1}) \leq 1-\frac{(\epsilon^\prime)^2}{2}   $ 
  
  for all $k_n \leq l \leq k_{n+1}$ with $n>0$.

\item[For the construction:] 
  
  \item[C6.] For every $m \geq k_n$ there is a finite set $Y_m \subseteq m$ such that
\begin{itemize}
    \item $Y_m \cap k_n = Z_n$
    \item $\epsilon^\prime \leq \frac{|Y_m \cap l|}{l} \leq 1- \epsilon^\prime $ for all $k_n \leq l \leq m$.
    \item $\displaystyle  \frac{d_{l}(\bigcap_{i \in E} B_i \cap Y_m)}{\prod_{i\in E}d_{l}(B_i)d_{l}(Y_m)} \in \left( 1-\frac{\delta_n}{15}, 1+ \frac{\delta_n}{15} \right)$ for all $k_n \leq l \leq m$.
\end{itemize}

\end{description}

\begin{description}

   \item[Intervals:] There is $k_n^\prime \geq k_n$ such that
    
   \item[C7.] $\displaystyle |Z_{n+1} \cap [k_n^\prime, k_{n+1}) | > (1-\epsilon^\prime)|Z_{n+1} \cap k_{n+1} |$

  \item[C8.] $\displaystyle r_n \Vdash Z_{n+1} \cap [k_n^\prime, k_{n+1}) = \dot{X}_\alpha \cap [k_n^\prime, k_{n+1}) $   
\end{description}

\end{multicols}

It is clear that (C1)-(C8) taken together for all $n<\omega$ imply that $A_\alpha:= \bigcup_{n<\omega} Z_n$ is moderate and satisfies the property (i) and (ii) with the intervals $I_n:=[k_n^\prime, k_{n+1})$. Now, we are going to begin our construction.

For $n=0$, choose any $k_0$, $q^\prime_0 \leq q_0$ and  $Z_0$ such that (C1), (C2), (C4) and (C6) hold. There is nothing to show yet for (C3), (C5), (C7) and (C8). 

Now assume that we have obtained $k_n \in \omega$, $q_n^\prime \leq q_n$ and $Z_n$ such that (C1), (C2), (C4) and (C6) hold for $n$. We will construct $r_n \leq q_n^\prime$, $k_{n+1} \in \omega$, $q_{n+1}^\prime \leq q_{n+1}$ and $Z_{n+1}$ such that (C3), (C5), (C7) and (C8) hold for $n$ and (C1), (C2), (C4) and (C6) hold for $n+1$. First, note that $Z_{n}$, $\frac{\delta_n}{15}$ and $\epsilon^\prime$ satisfies the conditions of lemma \ref{equalizessizes} since (C1) and (C6) holds for $n$. Then, we can find  $k_n^\prime \geq k_n$, a condition $q_n^{\prime\prime} \leq q_n^\prime$ and a finite set $Z_n^\prime \subseteq k_n^\prime $ such that 
\begin{itemize}
    \item $Z^\prime_n \cap k_n = Z_n$.
    \item $\displaystyle \frac{(\epsilon^\prime)^2}{2} \leq \frac{|Z^\prime_n \cap l|}{l} \leq 1- \frac{(\epsilon^\prime)^2}{2} $ for all $k_n \leq l < k^\prime_n$.
    \item $\displaystyle  \frac{d_{l}(\bigcap_{i \in E} B_i \cap Z^\prime_n)}{\prod_{i\in E}d_{l}(B_i)d_{l}(Z^\prime_n)} \in \left( 1- \delta_n, 1+ \delta_n \right)$ for all $k_n \leq l < k^\prime_n$ and for any $E \subseteq n$.
    \item $\displaystyle \epsilon^\prime \leq \frac{|Z^\prime_n \cap k^\prime_n|}{k^\prime_n} \leq 1- \epsilon^\prime$ 
    \item $\displaystyle  \frac{d_{k^\prime_n}(\bigcap_{i \in E} B_i \cap Z^\prime_n)}{\prod_{i\in E}d_{k^\prime_n}(B_i)d_{k^\prime}(Z^\prime_n)} \in \left( 1- \frac{\delta_n}{15}, 1+ \frac{\delta_n}{15} \right)$ for any $E \subseteq n$.
    \item $q_n^{\prime\prime} \Vdash \displaystyle \frac{|Z^\prime_n \cap k^\prime_n|}{|\dot{X}_\alpha \cap k^\prime_n|}, \frac{|\dot{X}_\alpha \cap k^\prime_n|}{|Z^\prime_n \cap k^\prime_n|} \in \left(1 -\frac{\delta_n}{15}, 1 + \frac{\delta_n}{15}\right)$
\end{itemize}
Now, find a condition $q_{n+1}^\prime \leq q_{n+1}$ and $k_n^{\prime\prime} \geq k_n^\prime$ such that for all  $E \subseteq \{0, ..., n+1\}$ we have that
\[
\displaystyle q_{n+1}^\prime \Vdash  \frac{d_{l}(\bigcap_{i \in E} B_i \cap \dot{X}_\alpha)}{\prod_{i\in E}d_{l}(B_i)d_{l}(\dot{X}_\alpha)} \in \left( 1-\frac{\delta_{n+1}}{15}, 1+ \frac{\delta_{n+1}}{15}  \right) 
\]
and
\[
\displaystyle q_{n+1}^\prime \Vdash  \frac{\prod_{i\in E}d_{l}(B_i)d_{l}(\dot{X}_\alpha)}{d_{l}(\bigcap_{i \in E} B_i \cap \dot{X}_\alpha)} \in \left( 1-\frac{\delta_{n+1}}{15}, 1+ \frac{\delta_{n+1}}{15}  \right) 
\]
for all $l\geq k_n^{\prime\prime}$. This is possible since the assumption in property (ii) is true. This implies that (C2) holds for $n+1$. Next, we find $r_n \leq q_n^{\prime\prime}$ and a sufficiently large $k_{n+1} > k_n^{\prime\prime}$ such that for all  $E \subseteq \{0, ..., n+1\}$ we have
\begin{equation}\label{equaton for independence}
\displaystyle r_n \Vdash  \displaystyle  \frac{d_{k_{n+1}}(\bigcap_{i \in E} B_i \cap (Z_n^\prime \cup (\dot{X}_\alpha \cap [k_n^\prime, k_{n+1}))) )}{\prod_{i\in E}d_{k_{n+1}}(B_i)d_{k_{n+1}}(Z_n^\prime \cup (\dot{X}_\alpha \cap [k_n^\prime, k_{n+1})))} \in \left( 1- \frac{\delta_{n+1}}{15}, 1+ \frac{\delta_{n+1}}{15} \right)   
\end{equation}
\begin{equation}\label{equation for moderacy}
\displaystyle r_n \Vdash \epsilon_0 \leq \frac{\dot{X}_\alpha \cap [k_n^\prime, k_{n+1})}{k_{n+1}-k_n^\prime} \leq  1- \epsilon_0  
\end{equation}
and
\[
\displaystyle r_n \Vdash |\dot{X}_\alpha \cap [k_n^\prime, k_{n+1})|   > (1-\epsilon^\prime)|\dot{X}_\alpha \cap [k_n^\prime, k_{n+1}) |    + (1-\epsilon^\prime) |Z_{n}^\prime |
\]
Besides, we can assume that $r_n$ decides $\dot{X}_\alpha \cap k_{n+1}$. Specifically, there is $X_n \subseteq [k_n^\prime, k_{n+1})$ such that $r_n \Vdash X_n=\dot{X}_\alpha \cap [k_n^\prime, k_{n+1})$. All this is also possible because of the assumption in property (ii), lemma \ref{stabilizepieces} and the following is forced
\[
\begin{array}{rl}
      \displaystyle \limsup\limits_{l\rightarrow \infty} \frac{|\dot{X}_\alpha \cap l|}{l} &= \displaystyle \limsup\limits_{l\rightarrow \infty} \frac{|\dot{X}_\alpha \cap [k_n^\prime, l)|}{l-k_n^\prime} \\
      \displaystyle \liminf\limits_{l\rightarrow \infty} \frac{|\dot{X}_\alpha \cap l|}{l} &= \displaystyle \liminf\limits_{l\rightarrow \infty} \frac{|\dot{X}_\alpha \cap [k_n^\prime, l)|}{l-k_n^\prime}
\end{array}
\]
Let $Z_{n+1}:=Z_n^\prime \cup X_n$. Note that (C7) and (C8) hold for $n$, and (C1) holds for $n+1$ by definition of $Z_{n+1}$. Let us see that the remaining conditions hold. Consider $W$ such that $r_n \Vdash W=\dot{X}_\alpha \cap k_{n+1}$.

\begin{description}

\item[Condition (C5).] By construction $
\frac{(\epsilon^\prime)^2}{2} \leq \frac{|Z_n^\prime \cap l|}{l} \leq 1- \frac{(\epsilon^\prime)^2}{2}
$ for $k_n \leq l < k_n^\prime$ and $\displaystyle \epsilon^\prime \leq \frac{|Z^\prime_n \cap k^\prime_n|}{k^\prime_n} \leq 1- \epsilon^\prime$. Since 
\[
\epsilon^\prime \leq  \frac{|W \cap k_n^\prime|}{k_n^\prime}, \frac{|W \cap l|}{l} \leq 1-\epsilon^\prime
\]
for any $l$ with $k_n^\prime \leq l \leq k_{n+1}$, by lemma \ref{three times} we get
\[
\frac{(\epsilon^\prime)^2}{2} \leq \displaystyle d_{l}(Z_{n+1}) \leq 1- \frac{(\epsilon^\prime)^2}{2}
\]
Thus (C5) holds for $n$.

\item[Condition (C4).]  Note that by Eq. (\ref{equation for moderacy}) and (C8) for $n$ we have
\[
\epsilon_0   \leq \frac{X_n}{k_{n+1}-k_n^\prime} \leq 1-\epsilon_0 
\]
Thus, we can apply lemma \ref{union of intervals} to $R:= k_n^\prime$, $S:=[k_n^\prime, k_{n+1})$, $A=Z_n^\prime$ and $B=X_n$ to obtain
\[
\epsilon_0 - \frac{1}{c} \leq \frac{Z_{n}^\prime\cup X_n}{k_n^\prime \cup [k_n^\prime, k_{n+1})} = d_{k_{n+1}}(Z_{n+1}) \leq 1-\epsilon_0 + \frac{1}{c}
\]
where $c=\frac{k_{n+1}-k_n^\prime}{k_n^\prime}$. Note that we can choose $k_{n+1}$ sufficiently large as to guarantee $\epsilon^\prime < \epsilon_0 - \frac{1}{c}$, so (C4) holds for $n+1$.    

\item[Condition (C6).] Let $l \geq k_{n+1}$. Let $X_l$ be any finite set that is a witness of $\dot{X}_\alpha \cap [k_n^\prime, l)$, that is, there is a condition $r \leq r_n$ such that $r \Vdash X_l = \dot{X}_\alpha \cap [k_n^\prime, l)$. Then, by the previous arguments, $Y_l:= Z_n^\prime \cup X_l = Z_{n+1} \cup X_l \setminus X_n$ satisfies the inequalities of the condition (C6). Thus (C6) holds for $n+1$.

\item[Condition (C3).] By construction we already have
\[
\displaystyle  \frac{d_{l}(\bigcap_{i \in E} B_i \cap Z_{n+1})}{\prod_{i\in E}d_{l}(B_i)d_{l}(Z_{n+1})} \in \left( 1- \delta_n, 1+ \delta_n \right) 
\]
for all $k_n \leq l < k^\prime_n$ and $E \subseteq \{0, ..., n\}$. Besides, by (C2) and (C8) for $n$ we have  
\begin{itemize}
    \item $\displaystyle  \frac{d_{k_n^\prime}(\bigcap_{i \in E} B_i \cap Z_n^\prime)}{\prod_{i\in E}d_{k_n^\prime}(B_i)d_{k_n^\prime}(Z_n^\prime)} \in \left( 1-\frac{\delta_n}{15}, 1+ \frac{\delta_n}{15}  \right)$ for any $E \subseteq \{0, ..., n\}$.
    \item $\displaystyle\frac{d_{l}(\bigcap_{i \in E} B_i \cap W)}{\prod_{i\in E}d_{l}(B_i)d_{l}(W)}, \frac{\prod_{i\in E}d_{l}(B_i)d_{l}(W)}{d_{l}(\bigcap_{i \in E} B_i \cap W)} \in \left( 1-\frac{\delta_n}{15}, 1+ \frac{\delta_n}{15}  \right) $ for $l \geq k_n^\prime$ and $E \subseteq \{0, ..., n\}$.
    \item $\displaystyle \frac{|Z^\prime_n \cap k^\prime_n|}{|W \cap k^\prime_n|}, \frac{|W \cap k^\prime_n|}{|Z^\prime_n \cap k^\prime_n|} \in \left(1 -\frac{\delta_n}{15}, 1 + \frac{\delta_n}{15}\right)$
\end{itemize}
 Thus, apply lemma \ref{joinpieces} to $Z_n^\prime$ and $W$ to get
 \[
 \displaystyle  \frac{d_{l}(\bigcap_{i \in E} B_i \cap Z_{n+1})}{\prod_{i\in E}d_{l}(B_i)d_{l}(Z_{n+1})} \in \left( 1- \delta_n, 1+ \delta_n \right) 
 \]
 for $k_n^\prime \leq l \leq k_{n+1} $ and $E \subseteq \{0, ..., n\}$. Therefore, (C3) holds for $(n)$.
 
\end{description}

\end{proof}

\bibliographystyle{amsalpha}
\bibliography{biblio}

\hspace{2mm}

\noindent \text{\small David Valderrama,} \textsc{\small Universidad de Los Andes (Bogotá).} 

\noindent \textit{\small E-mail address:} \text{\small \href{d.valderramah@uniandes.edu.co}{d.valderramah@uniandes.edu.co}}

\end{document}